\documentclass[a4paper,	11pt]{article}

\usepackage{bigints}

\usepackage[english]{babel} 
\usepackage[T1]{fontenc} 
\usepackage[babel=true]{csquotes} 
\usepackage{aeguill} 

\usepackage{latexsym}
\usepackage{comment}
\usepackage{natbib}

\usepackage{appendix}
\usepackage[normalem]{ulem} 
\usepackage{hyperref}
\usepackage{upgreek}
\usepackage{amssymb}
\usepackage{amsmath}
\usepackage{amsmath,mathrsfs}
\usepackage{amsthm} 
\usepackage{mathtools}
\usepackage{graphicx}
\usepackage{empheq}
\usepackage{color}
\usepackage{empheq}
\usepackage{float}
\usepackage{caption}
\usepackage{subcaption}
\usepackage{xcolor}
\usepackage{adjustbox}
\usepackage{listings}
\usepackage{algorithm}
\usepackage{algpseudocode}
\usepackage[algo2e,ruled,vlined]{algorithm2e} 
\usepackage{algorithmicx}
\usepackage{bbm}
\usepackage{enumitem}
\usepackage{mathrsfs}

\definecolor{dkgreen}{rgb}{0,0.6,0}
\definecolor{gray}{rgb}{0.5,0.5,0.5}
\definecolor{mauve}{rgb}{0.58,0,0.82}

\def\1{\mathbf{1}}
\def \d{\mathrm{d}} 
\def\eps{\varepsilon} 

\def\d{{\rm d}}

\def \N{\mathbb{N}}

\def \E{\mathbb{E}}
\def \F{\mathbb{F}}

\def \L{\mathbb{L}}
\def \P{\mathbb{P}}
\def \Q{\mathbb{Q}}

\def \R{\mathbb{R}}
\def \S{\mathbb{S}}
\def \tr{\text{tr}}
\def \Law{\text{Law}}
\def \dist{\text{dist}}
\def \conc{\text{conc}}
\newcommand{\conv}{\mathrm{conv}}

\makeatletter \@addtoreset{equation}{section}

\usepackage{amsmath, amssymb}
\usepackage{graphicx}
\usepackage{color}
\usepackage{textcomp}
\usepackage{graphics}
\usepackage{float}
\usepackage{empheq}
\usepackage{caption}
\usepackage{subcaption}
\usepackage{mathrsfs}

\usepackage{comment}

\usepackage{geometry}
\usepackage{stmaryrd}

\AtBeginDocument{}
\AtBeginDocument{}

\def \ep{\hbox{ }\hfill{ ${\cal t}$~\hspace{-5.1mm}~${\cal u}$}}

\def\b*{\begin{eqnarray*}}
\def\e*{\end{eqnarray*}}

\newcommand{\be}{\begin{equation}}
\newcommand{\ee}{\end{equation}}

\newtheorem{theorem}{Theorem}[section]
\newtheorem{lemma}[theorem]{Lemma}   
\newtheorem{proposition}[theorem]{Proposition}

\newtheorem{rem}[theorem]{Remark}

\mathtoolsset{showonlyrefs}

\newcommand{\bea}{\begin{eqnarray}}
\newcommand{\bes}{\begin{subEquations}}
\newcommand{\ees}{\end{subEquations}}
\newcommand{\bgt}{\begin{gather}}
\newcommand{\egt}{\begin{gather}}
\newcommand{\eea}{\end{eqnarray}}
\newcommand{\beaa}{\begin{eqnarray*}}
\newcommand{\eeaa}{\end{eqnarray*}}

\def \Cc{{\cal C}}

\def \Fc{{\cal F}}

\def \Pc{{\cal P}}

\def \Mc{{\cal M}}
\def \Nc{{\cal N}}

\def \Tc{{\mathbf T}}

\def \Vc{{\cal V}}

\def \msY{{\cal Y}}

\def \Leb{\text{Leb}}

\def \msX{\mathscr{X}} 
\def \msY{\mathscr{Y}} 

\def \msD{\mathscr{D}} 

\def \d{\mathrm{d}} 
\def\eps{\varepsilon}

\def\b*{\begin{eqnarray*}}
\def\e*{\end{eqnarray*}}
\def\be{\begin{eqnarray}}
\def\ee{\end{eqnarray}}

\numberwithin{equation}{section}

\begin{document}

\title{Bridging Schr\"odinger and Bass: A Semimartingale Optimal Transport Problem with Diffusion Control}

\author{ 
Pierre HENRY-LABORDERE \footnote{QubeRT. Email: phl@hotmail.com}
\and Grégoire LOEPER \footnote{BNPP  and  Monash University. Email: gregoire.loeper@bnpparibas.com}  
\and Othmane MAZHAR \footnote{LPSM, Sorbonne Université and Universit\'e Paris Cit\'e. This author was supported by the BNP-Paribas 
Chair ``Futures of Quantitative Finance''. Email: othmane.xx90@gmail.com}
\and Huy\^en PHAM  \footnote{Ecole Polytechnique, CMAP. This author is supported by the Chair ``Financial Risks'', by FiME (Laboratory of Finance and Energy Markets), and the EDF–CACIB Chair ``Finance and Sustainable Development''. Email: huyen.pham@polytechnique.edu}
\and Nizar TOUZI \footnote{NYU Tandon School of Engineering. This author is partially supported by NSF grant $\#$DMS-2508581. Email: nizar.touzi@nyu.edu}
}

\maketitle

\begin{abstract}
We study a semimartingale optimal transport problem interpolating between the Schrödinger bridge and the stretched Brownian motion associated with the Bass solution of the Skorokhod embedding problem. The cost combines an entropy term on the drift with a quadratic penalization of the diffusion coefficient, leading to a stochastic control problem over drift and volatility.
 
We establish a complete duality theory for this problem, despite the lack of coercivity in the diffusion component. In particular, we prove strong duality and dual attainment, and derive an equivalent reduced dual formulation in terms of a variational problem over terminal potentials.

Optimal solutions are characterized by a coupled Schr\"odinger–Bass bridge system, involving a backward heat potential and a transport map given by the gradient of a convex function. This system interpolates between the classical Schr\"odinger system and the Bass martingale transport.
Our results provide a unified framework encompassing entropic and martingale optimal transport, and yield a variational foundation for data-driven diffusion models. 
 
\end{abstract}

\vspace{5mm}


\vspace{5mm}


 \tableofcontents 

\section{Introduction}

Classical Schrödinger bridges and martingale optimal transport provide two fundamental constructions of optimal couplings of probability measures via continuous-time stochastic processes. In the Schrödinger problem, one fixes a reference Brownian motion and prescribes initial and terminal distributions, then minimises the relative entropy of the law of the process with respect to the reference path measure. This leads to a stochastic control problem with controlled drift and fixed diffusion coefficient, admitting a well-posed dual formulation and a characterisation in terms of the Schrödinger system, see \cite{Leo14},  \cite{chenetal21}, \cite{nutz22}.  In contrast, martingale optimal transport imposes a martingale constraint and allows for non-trivial control of the diffusion coefficient. The Bass problem, whose optimiser is the stretched Brownian motion, consists in prescribing marginal laws of a continuous martingale while minimising a deviation from Brownian motion; see \cite{ConzeHenryLabordere2019,bacetal23b,acciaio2023localvolatility}. This formulation is closely related to martingale Benamou–Brenier transport \cite{bacetal20} and plays a central role in model-independent finance. In contrast with the Schrödinger problem, it exhibits degeneracies and lacks a regular dual structure.

The aim of this paper is to introduce and analyse a class of optimal semimartingale transport problems that interpolate between these two regimes. Given $\mu_0,\mu_T\in\Pc_2(\R^d)$, we consider It\^o processes
$X=(X_t)_{0\le t\le T}$ with $X_0\sim\mu_0$, $X_T\sim\mu_T$, and we minimise
a functional which penalises both the drift and the volatility of $X$,
weighted by a parameter $\beta$ $>$ $0$. We denote the resulting value by ${\rm SBB}(\mu_0,\mu_T)$
and refer to this optimisation problem as the Schr\"odinger–Bass bridge
(SBB) problem. As $\beta\to\infty$, one recovers the Schr\"odinger bridge, while as $\beta\to 0$, the problem reduces to a Bass-type martingale transport. The SBB problem therefore provides a unified stochastic control formulation of entropic and martingale transport.

 The main difficulty in the analysis stems from the lack of coercivity in the diffusion variable. In contrast with the framework of controlled diffusions studied in \cite{TanTouzi2013}, \cite{guo2021path},  the cost functional does not penalise large diffusion coefficients in a coercive manner, and standard compactness and duality arguments do not apply directly.

 Our first main result establishes strong duality for the SBB problem. We prove that the primal stochastic control problem admits a minimiser and that its value coincides with that of a dual problem over functions $v\in C^{1,2}$ solving a fully nonlinear Hamilton–Jacobi–Bellman equation under a curvature constraint $D^2 v<\beta I_d$. Moreover, the HJB equation admits an explicit representation in terms of a terminal potential $\phi$ through a quadratic inf–convolution operator, leading to an equivalent static dual formulation of Donsker–Varadhan type. We show that the  static dual problem attains its supremum over a suitable relaxed space. 

Our second main result is a complete characterisation of the solution to SBB. We show that the optimal SBB bridge is encoded by a triplet $(h,\nu,\msY)$ solving a coupled system, which we call the Schr\"odinger–Bass bridge system. The functions $h$ and $\nu$ solve backward and forward heat equations, while the map $\msY$ is the gradient of a quadratic inf–convolution transform. The marginal laws $(\mu_t)_{t}$ of the optimal process $(X_t)_t$ satisfy
\begin{align}
\msY_t \#\mu_t & = \; h_t \nu_t, \quad \mbox{ i.e. } \quad \mu_t \;  = \: \msX_t\#(h_t \nu_t), \quad 0 \leq t \leq T,   
\end{align}
where $\#$ denotes the push-forward operation,  $\msX_t$ $=$ $\msY_t^{-1}$ is the inverse of the homeomorphism $\msY_t$, and is given by the gradient of a convex function:
\begin{align}
\msX_t(y) &= \; \nabla_y \Big( \frac{|y|^2}{2} + \frac{1}{\beta} \log h_t(y) \Big) \; = \; y + 
\frac{1}{\beta} \nabla_y \log h_t(y), 
\end{align}
thus $\msX_t$ is expressed in terms of the score of the potential density $h_t$. 
This system reduces to the classical Schr\"odinger system as $\beta\to\infty$ and to the Bass martingale transport system as $\beta\to 0$.

Finally, we identify the structure of the optimal semimartingale. The optimal drift and diffusion coefficients admit explicit feedback representations in terms of $(h,\msY)$. Moreover, after the change of variables $Y_t$ $=$ $\msY_t(X_t)$, $0\leq t\leq T$, the process $Y$ is a Schr\"odinger bridge diffusion and becomes a Brownian motion under an equivalent change of measure. This yields a stretched Brownian representation of the optimal SBB bridge, extending the Bass construction from martingales to general semimartingales. 

The paper is organised as follows. Section \ref{sec:SBB} introduces the SBB problem and derives a first dual formulation. Section \ref{sec:main} states the main results. Sections \ref{sec:duality} and \ref{sec:dualreduced} establish duality and analyse the reduced dual problem. Sections \ref{sec:dualattainment} and \ref{sec:primal} prove dual and primal attainment and derive the structure of optimal solutions.
\footnote{During the final stage of the preparation of this paper, G. Pammer  brought to our attention his work in \cite{pammer} with M. Hasenbichler and S. Thonhauser, motivated by early presentations of the results reported here, in which they analyze the same problem through the lens of weak  optimal transport.} 

\section{Schr\"odinger–Bass bridge optimal transport} \label{sec:SBB}

\subsection{Problem formulation} 
\label{sec: SB}

Let $T>0$ be some finite maturity and let $\Omega=C([0,T],\R^d)$ be the canonical space equipped with its canonical filtration $\F$ $=$ $(\Fc_t)_t$ and canonical process $X$, i.e. $X_t(\omega)=\omega(t)$ for all $\omega\in\Omega$, $t\in [0,T]$. We denote by $\Pc$ the set of probability measures $\P$ on  $\Omega$ under which $X$ has the diffusion decomposition
 \begin{align} \label{decdiff} 
 X_t &= \   X_0 + \int_0^t \alpha_s^\P \d s + \int_0^t \sigma_s^\P \d W_s^\P, \qquad  t \in [0,T], \  \P-\mbox{a.s.} 
 \end{align} 
 for some $d$-dimensional $\P-$Brownian motion $W^\P$, and some characteristics $\gamma^\P=(\alpha^\P,\sigma^\P)$ that are $\F$-progressively measurable processes valued in $\R^d\times\S_+^d$,  satisfying $\int_0^T |\alpha_t^\P|\d t+\int_0^T |\sigma_t^\P|^2 \d t<\infty$, $\P$-a.s. 
 
Given two distributions $\mu_0,\mu_T\in\Pc_2(\R^d)$, we introduce the subset of transport plans  
  \begin{equation}
 \Pc(\mu_0,\mu_T) \ =\  \big\{ \P \in \Pc:  \P \circ X_0^{-1} = \mu_0, \  \mbox{ and } 
 \  \P \circ X_T^{-1} = \mu_T \big\}.
  \end{equation}
Given a parameter $\beta>0$, and denoting by $I_d$ the identity matrix in $\R^{d\times d}$, we introduce the cost function
\begin{equation}\label{def:c}
c(a,b):=\frac12|a|^2+\frac{\beta}2|b - I_d |^2,
~a\in\R^d,~b\in\S_+^d.
\end{equation}
Our objective in this paper is to derive a full characterization of a critical transport plan obtained through the following optimal transport problem on the canonical path space:
 \begin{equation}\label{SBB}
{\rm SBB}(\mu_0,\mu_T) 
:=
\inf_{\P\in\Pc(\mu_0,\mu_T)} J^0(\P),
~\mbox{with}~
J^0(\P) 
:=
\E^{\P} \Big[  \int_0^Tc(\gamma_t^\P)\d t \Big].
 \end{equation}
Here the SBB acronym stands for {\it Schr\"odinger–Bass bridge}, and is justified by the two extreme cases $\beta\to 0$ and $\beta\to\infty$:  
\begin{itemize}
\item Formally, when $\beta$ goes to infinity, we constrain to those transport plans in the subset $\Pc_0(\mu_0,\mu_T):=\{\P\!\in\!\Pc(\mu_0,\mu_T):\sigma^\P=I_d~\mbox{on}~[0,T]\}$; the problem SBB reduces to the classical {\it Schr\"odinger bridge} problem which aims to find the closest transport plan $\P\in\Pc_0(\mu_0,\mu_T)$ to the Brownian path measure with initial law $\mu_0$ in the sense of the relative entropy (Kullback-Leibler) distance, see, e.g. \cite{Leo14},
$$
{\rm SB}(\mu_0,\mu_T) 
:=
\inf_{\P\in\Pc_0(\mu_0,\mu_T)} \E^{\P} \Big[  \frac12\int_0^T|\alpha_t^\P|^2\d t \Big].
$$
\item In the other extreme case, by dividing the criterion $J^0$ by $\beta$, and sending $\beta$  to zero,  we formally constrain the drift coefficient to be zero, and then we are looking for a continuous martingale, namely the stretched Brownian motion/Bass martingale in the appropriate setting, which is closest to Brownian motion according to the quadratic volatility norm, under marginal constraints. This martingale transport problem was studied in Backhoff-Veraguas, Beiglb\"ock, Huesmann \& K\"allblad 
\cite{bacetal20}, Backhoff-Veraguas, Schachermayer \& Tschilderer \cite{bacetal23b}, who called its solution the {\it Stretched Brownian motion} and which is intimately related to the Bass  martingale with degenerate initial law. The latter was motivated by calibration problems in financial engineering, and was first studied by Conze \& Henry-Labord\`ere \cite{ConzeHenryLabordere2019} and Acciaio, Pammer \& Marini \cite{acciaio2023localvolatility}. 
\end{itemize}
The solution (when it exists) $\hat{\P}$ of ${\rm SBB}(\mu_0,\mu_T)$ is called the Schr\"odinger–Bass bridge (SBB) transport plan. We start with some initial considerations which will underpin the subsequent analysis of the problem SBB.

\begin{rem}\label{rem:lin} The linear coupling $\P^\pi_{\rm lin}\in\Pc(\mu_0,\mu_T)$ {\rm is defined for all static coupling $\pi\in\Pi(\mu_0,\mu_T)$ as follows. Let $(Y_0,Y_T)$ be a random vector on some probability space, with law $\pi$, and let $\P^\pi_{\mathrm{lin}} := \text{Law}(Y)$ be the law of the linear interpolation $Y_t :=\frac{T-t}{T}Y_0+\frac{t}{T}Y_T,$ $t\in[0,T]$. Back to our notations on the canonical space, we observe that $\d X_t=\alpha^{\P^\pi_{\mathrm{lin}}}_t\d t+0\ \d W_t,$ $\P^\pi_{\mathrm{lin}}-$a.s. with corresponding
$\alpha^{\P^\pi_{\mathrm{lin}}}_t:=\frac{X_T-X_0}{T},\  0\le t\le T$,
and, for $t>0$, equivalently
$\alpha^{\P^\pi_{\mathrm{lin}}}_t=\frac{X_t-X_0}{t}$ satisfying $J^0(\P^\pi_{\mathrm{lin}})=\frac12\E\int_0^T|\alpha^{\P^\pi_{\mathrm{lin}}}_t|^2dt + \frac{\beta d T}{2}=\frac12T^{-1}\E|X_T-X_0|^2+ \frac{\beta d T}{2}<\infty$. Consequently $\P^\pi_{\mathrm{lin}}\in\Pc(\mu_0,\mu_T)$.
}
\end{rem}

\begin{lemma}\label{lem:supX2}
${\rm SBB}(\mu_0,\mu_T)\in[0,\infty)$, and $\E^\P[\sup_{t\le T}|X_t|^2]<\infty$ for every $\P\in\Pc(\mu_0,\mu_T)$ with $J^0(\P)<\infty$.
\end{lemma}

\proof
By Remark \ref{rem:lin}, we have ${\rm SBB}(\mu_0,\mu_T)\le J^0(\P_{\mathrm{lin}})<\infty$. Next, for $\P\in\Pc(\mu_0,\mu_T)$ with $J^0(\P)<\infty$, it follows from the BDG inequality that $\E^\P[\sup_{t\le T}|X_t|^2]\le C(1+\E^\P[|X_0|^2]+\E^\P\!\int_0^T\big(|\alpha_t^\P|^2+|\sigma_t^\P|^2\big)\d t)\le C'(1+\E^\P[|X_0|^2]+J^0(\P))<\infty$, for some constants $C,C'>0$.
\ep

\subsection{A first dual problem}

The problem SBB is a semimartingale optimal transport problem in the sense of \cite{TanTouzi2013}, see also \cite{guo2021path}. Notice however that our cost function does not satisfy the coercivity condition in $b$ that is required in \cite{TanTouzi2013}. Despite this, we shall justify that the linear programming duality still holds in our setting and we may then express the constrained control problem SBB in terms of the penalized unconstrained control problem.

\medskip
Let $w:=1+|.|^2$ and let $\Cc_w:=\{\psi\in C^0(\R^d,\R):|\frac{\psi}{w}|_\infty<\infty\}$ be the collection of all continuous functions with quadratic growth on $\R^d$. As $\mu_T\in\Pc_2(\R^d)$, we have
$$
\Pc(\mu_0,\mu_T)
=
\big\{\P\in\Pc(\mu_0): \E^\P[\psi(X_T)]=\mu_T(\psi)
                                   ~\mbox{for all}~\psi\in \Cc_w
\big\},
$$
with $\Pc(\mu_0):=\cup_{\nu\in\Pc_2(\R^d)}\Pc(\mu_0,\nu)$. We may then rewrite our problem as the penalized control problem
\begin{equation}\label{Jpsi}
{\rm SBB}(\mu_0,\mu_T) 
=
\inf_{\P\in\Pc(\mu_0)}
\sup_{\psi\in\Cc_w} \mu_T(\psi) + J^\psi(\P),
~J^\psi(\P):=\E^{\P} \Big[\!-\psi(X_T)+\!\!\int_0^T\!\!\!c(\gamma_t^\P)\d t \Big].
\end{equation}
Following the standard terminology in optimal transport theory, we call the penalization maps $\psi$ {\it potential functions}.
By the trivial inequality $\inf_{\P\in\Pc(\mu_0)}
\sup_{\psi\in\Cc_w}\ge
\sup_{\psi\in\Cc_w}\inf_{\P\in\Pc(\mu_0)}$, this provides the weak duality inequality
\begin{equation}\label{weakduality1}
{\rm SBB}(\mu_0,\mu_T) 
\ge 
\sup_{\psi\in\Cc_w} 
\mu_T(\psi)
+\inf_{\P\in\Pc(\mu_0)}J^\psi(\P).
\end{equation}
We rewrite the minimization on the right-hand side as a standard stochastic control problem by expressing it as the integration of the value function of a standard stochastic control problem with respect to $\mu_0$ . To do this, we denote for all $\P\in\Pc(\mu_0)$ by $\{\P^x\}_{x\in\R^d}$ a regular conditional law of $\P$ given $X_0$, and we write by the tower property that
$
J^\psi(\P)=\int J^\psi(\P^x)\mu_0(\d x)\ge\int \inf_{\Q\in\Pc(\delta_x)}J^\psi(Q)\mu_0(\d x)
$. Combining with \eqref{weakduality1}, we see that
\begin{equation}\label{weakduality}
{\rm SBB}(\mu_0,\mu_T) 
\ \ge\  
\Vc(\mu_0,\mu_T)
\ :=\ 
\sup_{\psi\in\Cc_w} 
\mu_T(\psi)
-\mu_0(V^\psi_0),
\end{equation}
where $V^\psi_0$ is the value function of a standard stochastic control problem:
\begin{equation}\label{Vpsi}
V^\psi_0(x)
:=
\sup_{\P\in\Pc(\delta_x)}\!-J^\psi(\P)
=
\sup_{\P\in\Pc(\delta_x)}
\E^\P\Big[\psi(X_T)-\int_0^Tc(\gamma^\P_t)\ \d t\Big],
~~x\in\R^d.
\end{equation}
The penalized optimisation problem $\Vc(\mu_0,\mu_T)$ is our first dual problem. Our primary objective in this paper is to prove that there is no duality gap in \eqref{weakduality}. In addition, we shall prove existence for both primal and dual problems, together with a complete characterization of the corresponding solution which exhibits a perfect interpolation between the well-known Schr\"odinger bridge system and the Bass-Brenier transport map.

\subsection{HJB equation}

We denote by $Q_T:=[0,T)\times\R^d$ the time-space parabolic domain, and $\overline Q_T:=[0,T]\times\R^d$ its closure.
By standard optimal control theory, the HJB equation corresponding to the problem $V^\psi_0$ is:
\begin{equation}
\label{HJB}
\partial_tv+H(Dv,D^2v)
=0~\mbox{on}~Q_T,
~\mbox{and}~v(T,\cdot) =\psi~\mbox{on}~\R^d,
\end{equation}
where the Hamiltonian $H$ is given by 
\begin{align}
H(p,A)
&:=
\sup_{a\in\R^d,\ b\in\S_+^d}\Big\{a\!\cdot\!p+\frac12bb^\intercal\!:\!A-c(a,b)\Big\}
\\
&= \frac12|p|^2
      + \frac{\beta}{2}\  I_d\!:\![\beta(\beta I_d - A)^{-1}-I_d]
~\mbox{on}~{\rm dom}(H):=\{(p,A):A<\beta I_d\},
\end{align}
and $H=\infty$ outside ${\rm dom}(H)$. We record that the maximizers of the Hamiltonian are:
\begin{equation}
\widehat{a}(p)=p,~\widehat{b}(A)=\beta(\beta I_d-A)^{-1},
~(p,A)\in{\rm dom}(H).
\end{equation}
Notice that finiteness of $H(Dv,D^2
v)$ formally forces $D^2v<\beta I_d$, and we then expect that the solution of the HJB equation \eqref{HJB} exhibits a boundary layer in the sense that $\lim_{t\nearrow T}v(t,\cdot)-\frac{\beta}2|.|^2$ is concave. For this reason, it turns out that the set of potential maps can be reduced to the subset of all such {\it $\beta-$concave} maps:
$$
\Cc_w^{\rm conc}
:=\Cc_w\cap\Cc^{\rm conc}
~\mbox{where}~
\Cc^{\rm conc}:=\big\{\phi\in C^0(\R^d):\phi ~\mbox{is $\beta-$concave} \big\}.
$$
where $\phi$ is $\beta-$concave if  $\phi -\frac{\beta}{2}|\cdot|^2$ is concave. Similarly, we say that $\phi$ is $\beta-$convex if the map $\phi+\frac\beta{2}|\cdot |^2$ is convex, and we denote
$$
\Cc_w^{\rm conv}
:=\Cc_w\cap\Cc^{\rm conv},
~\mbox{with}~
\Cc^{\rm conv}:=\big\{\phi\!\in\! C^0(\R^d):\phi ~\mbox{is $\beta-$convex} \big\}
=-\Cc^{\rm conc}.
$$

\subsection{Dual potential maps and reduced dual problem}
\label{sec:dualpotential} 

Due to the particular choice \eqref{def:c} of the cost function, the HJB equation \eqref{HJB} turns out to be amenable to interesting manipulation leading to more explicit solution structure.   
We shall use extensively the Moreau transform, also called inf–convolution:
\b*
\mathbf{T}_\beta^+[\phi](x)
:=
\inf_{y\in\R^d} \phi(y)+\frac\beta{2}|x-y|^2,
&x\in\R^d.&
\e*
Notice that $\psi:=\mathbf{T}_\beta^+[\phi]$ is $\beta-$concave as $\mathbf{T}_\beta^+[\phi]-\frac\beta{2}|\cdot |^2$ is an infimum of affine functions. We also introduce the dual operator
$\mathbf{T}_\beta^-[\psi]:=-\mathbf{T}_\beta^+[-\psi]$.
We shall use the standard Moreau biconjugation property: if $\phi$ is proper, lower semicontinuous and $\beta$-convex, then
\begin{equation}\label{TbetaTbeta}
\mathbf T_\beta^-\circ\mathbf T_\beta^+[\phi]=\phi.    
\end{equation}
More generally, for arbitrary $\phi$ such that $\mathbf T_\beta^+[\phi]$ is finite-valued, the function
$\phi^\beta:=\mathbf T_\beta^-\circ\mathbf T_\beta^+[\phi]$
is $\beta$-convex, satisfies $\phi^\beta\le \phi$, and
$\mathbf T_\beta^+[\phi^\beta]=\mathbf T_\beta^+[\phi]$.
When $\phi=\mathbf{T}_\beta^-[\psi]$ for some potential $\psi$, we call $\phi$ a {\it dual potential map}.

\medskip
Our main characterization result of the SBB problem reduces the dual problem $\Vc(\mu_0,\mu_T)$ to the following optimisation problem over the set of dual potential maps:
\begin{align}\label{D}
\Vc_{\rm red}(\mu_0,\mu_T)
:=
\sup_{\phi\in \Cc^{\rm conv}_{w}}
\mathfrak{J}(\phi),
~~\mbox{with}~~&
\mathfrak{J}(\phi)
:=
\mu_T\big(\mathbf{T}_\beta^+[\phi]\big)
- \mu_0\big(\mathbf{T}_\beta^+[u_T^\phi]\big)
\\
&
e^{u_s^\phi(y)}:=\Nc_{s}*e^{\phi}(y),
~(s,y)\in\overline{Q}_T,
\nonumber
\end{align}
where $\Nc_s:=(2\pi s)^{-\frac{d}2}e^{{-\frac{|.|^2}{2s}}}$ is the heat kernel in $\R^d$, with the convention $\Nc_0*e^{\phi} = e^{\phi}$. The justification of this expression for the dual is clarified by the subsequent Lemma \ref{lem:HC_Moreau} {\it (3-b)} below, which connects $\mathbf{T}_\beta^+[u_T^\phi]$ to the HJB equation corresponding to the control problem \eqref{Vpsi}.

\begin{rem}\label{rem:dual}{\rm
The restriction of the dual functions to be $\beta-$convex in \eqref{D} can be relaxed. To see this, set
$\tilde\phi:=\mathbf{T}_\beta^-\circ\mathbf{T}_\beta^+[\phi]$.
By the above Moreau envelope property, $\tilde\phi$ is $\beta$-convex,
$\tilde\phi\le \phi$, and
$\mathbf{T}_\beta^+[\tilde\phi]=\mathbf{T}_\beta^+[\phi]$.
Hence $u_T^{\tilde\phi}\le u_T^\phi$, and therefore
$\mathfrak{J}(\tilde\phi)\ge\mathfrak{J}(\phi)$.
Thus the $\beta$-convexity restriction on the dual potential does not affect the dual maximization problem.
}
\end{rem}

\subsection{A relaxed dual formulation}

It turns out that the formulation \eqref{weakduality} of the dual problem $\Vc$ does not guarantee existence of an optimal potential map. For this reason, we need to introduce an appropriate relaxation so as to allow for existence while preserving the same value. First, from the previous considerations, we already know that it is sufficient to restrict the dual maps to those $\beta-$concave maps. As $\beta-$concave maps are bounded from above by a quadratic function, it remains to specify the asymptotic behavior of $\psi^-$ at infinity. Instead, we introduce the following relaxed dual space by enforcing an appropriate integrability condition:
\begin{equation}\label{barCconc}
\bar\Cc^{\rm conc}
:=
\Big\{\psi~\mbox{is finite-valued and $\beta-$concave}:\ \mu_0\Big(
 \mathbf{T}_\beta^+\Big[u_T^{\mathbf{T}_\beta^-[\psi]}\Big]\Big)<\infty
 \Big\},
\end{equation}
where $V^\psi_0$ for $\psi\in\bar\Cc^{\rm conc}$ is defined as in \eqref{Vpsi} with the expectation understood in the extended sense. The quantity
$\mu_T(\psi)-\mu_0(V^\psi_0)$ is also understood in the extended sense, with the convention that it is equal to $-\infty$ whenever the positive and negative parts are not both well-defined.

Our relaxed dual formulation is defined as
\begin{equation}\label{barVc}
\bar\Vc(\mu_0,\mu_T)
:=
\sup_{\psi\in\bar\Cc^{\rm conc}}
\big\{\mu_T(\psi)-\mu_0(V^\psi_0)\big\}.
\end{equation}

\section{Main results} \label{sec:main} 

The main objective of this paper is to establish strong duality between the primal problem ${\rm SBB}$ in \eqref{SBB} and the dual problem $\Vc$ in \eqref{weakduality}. To ensure existence, the latter is relaxed to \eqref{barVc}; moreover, it can be expressed in the reduced form $\Vc_{\rm red}$ defined in \eqref{D}.

\begin{theorem}\label{thm:main}
Let $\mu_0,\mu_T\in\Pc_2(\R^d)$. Then

\vspace{2mm}
\noindent {\rm (i)} ${\rm SBB}=\Vc=\bar\Vc=\Vc_{\rm red}\in[0,\infty)$ at $(\mu_0,\mu_T)$;

\vspace{2mm}
\noindent {\rm (ii)} Under the additional condition $\beta T>1$:
\begin{itemize}

\vspace{-2mm}
\item[{\rm (a)}] The reduced dual problem $\Vc_{\rm red}$ has a solution $\hat\phi\in \Cc^{\rm conv}_{w}$, which induces a solution $\hat\psi:=\mathbf{T}_\beta^+[\hat\phi]\in\bar\Cc^{\rm conc}$ of the dual problem $\bar\Vc$; 

\vspace{-2mm}
\item[{\rm (b)}] The primal problem {\rm SBB} has a solution $\hat\P\in\Pc(\mu_0,\mu_T)$;

\vspace{-2mm}
\item[{\rm (c)}] For all $t<T$, the map
$\msY_t:={\rm id}-\frac1\beta\nabla \mathbf{T}_\beta^+[u^{\hat\phi}_{T-t}]$
is a well-defined one-to-one map, and we denote its inverse by
$\msX_t:=\msY_t^{-1}$. At terminal time $t=T$, $\msY_T$ is understood as
the corresponding set-valued argmin map. Moreover, there exists a probability
measure $\hat\Q$ under which $Y$ is a Schr\"odinger bridge diffusion,

\vspace{-3mm}
\begin{equation}
Y_t = Y_0+\int_0^t \nabla u^{\hat\phi}_{T-s}(Y_s)\,\d s
+W_t^{\hat\Q},
\qquad 0\le t<T,
\end{equation}
\vspace{-3mm}

and the process $X_t:=\msX_t(Y_t)$, $0\le t<T$, admits a continuous
extension to $[0,T]$ satisfying $\hat\P=\Law_{\hat\Q}(X)$.
Equivalently, if $\mathbf R$ denotes the Brownian law with initial distribution
$m_0:=\msY_0\#\mu_0$, then:

\vspace{-3mm}
\begin{equation}
\frac{d\hat\Q}{d\mathbf R}\Big|_{\Fc_t}
= e^{u^{\hat\phi}_{T-t}(Y_t)-u^{\hat\phi}_T(Y_0)}, \qquad 0\le t<T.
\end{equation}

\vspace{-3mm}
\item[{\rm (d)}] There exists a non-zero $\sigma$-finite measure $\nu_0$ such that, setting
$\nu_t=\Nc_t\!*\!\nu_0$ and $m_T:=e^{\hat\phi}\nu_T$, the dual optimiser
$\hat\phi$ is characterized by the SBB system

\vspace{-3mm}
\begin{equation}
\frac{\d\msY_0\#\mu_0}{\d\nu_0} = e^{u^{\hat\phi}_{T}},
\qquad
\mu_T=\msX_T\#m_T.
\end{equation}

\end{itemize}
\end{theorem}

We report the proof of the duality (i) in Sections \ref{sec:duality} and \ref{sec:dualreduced}, and the remaining part (ii) in Sections  \ref{sec:dualattainment} and \ref{sec:primal}. To better understand the last characterization of our solution, recall that the Schr\"odinger bridge corresponds to the formal limit $\beta\to\infty$, where $\msY$ formally reduces to $\msY_t={\rm id}$ for all $t\in[0,T]$, while the Bass martingale corresponds to the formal limit $\beta\to0$, where $\hat\phi=0$. We observe here that the characterization (ii-c) in the last statement combines features of the Schr\"odinger bridge and the Bass martingale:
\begin{itemize}
\item Similar to the structure of the Schr\"odinger bridge, the flow of probability measures $(\nu_t)_{t\le T}$ satisfies the forward Fokker-Planck equation, while the map $(t,y)$ $\mapsto$ 
$e^{u_{T-t}^{\hat\phi}(y)}$ is a solution of the backward heat equation;
\item The map $\msY$ is the gradient of the convex map $\frac1\beta(\mathbf{T}_\beta^+[u^{\hat\phi}_{T-t}]-\frac\beta{2}|.|^2)$ is our substitute for the Brenier map in the Bass martingale component of our characterization. Its inverse $\msX$ is expressed as $\msX_t$ $=$ ${\rm id} + \frac1\beta\nabla u_{T-t}^{\hat\phi}$. 
\end{itemize}
We represent the SBB system in the standard graphical form adopted in the Schr\"odinger bridge literature as:
\begin{equation}
\begin{array}{ccccccl}
\mu_T
  & \!\!\!\!-\!\!\!-\!\!\!-\!\!\!\longrightarrow
  &\!\!\!\!m_T:=\msY_T \# \mu_T
  & \!\!\!\!-\!\!\!-\!\!\!-\!\!\!\longrightarrow
  &  \!\!\!\!{ \frac{dm_T}{d\nu_T} = e^{\hat\phi}}
  &  \!\!\!\!-\!\!\!-\!\!\!-\!\!\!\longrightarrow
  & \nu_T 
  \\
\Big\uparrow & & & & & &\Big\uparrow 
\\
\Big| & &\hspace{-25mm}\mbox{\colorbox{gray}{\color{white} ~~~~Bass~~~~}}& 
                         && \hspace{-15mm}\mbox{\colorbox{gray}{\color{white} Schr\"odinger}}
                         & \hspace{.6mm}\Big| 
\\
\Big| & & & & & & \hspace{.58mm}\Big|
  \\
\mu_0
  & \!\!\!\!-\!\!\!-\!\!\!-\!\!\!\longrightarrow
  & \!\!\!\!m_0:=\msY_0 \# \mu_0
  & \!\!\!\!-\!\!\!-\!\!\!-\!\!\!\longrightarrow
  &  \!\!\!\!\frac{dm_0}{d\nu_0} = e^{u^{\hat\phi}_T}
  & \!\!\!\!-\!\!\!-\!\!\!-\!\!\!\longrightarrow
  & \nu_0
\end{array}
\end{equation}
At terminal time, $\msY_T$ is understood as the argmin correspondence, and the rigorous terminal relation is
$\frac{dm_T}{d\nu_T}=e^{\hat\phi},
\ \mu_T=(\msX_T)_\#m_T$,
rather than a deterministic push-forward identity $m_T=\msY_T\#\mu_T$.

One can view the SBB solution as a stretched Schr\"odinger diffusion, i.e., a Bass transport of a Schr\"odinger bridge.

\section{Proof of the first duality results}  \label{sec:duality} 

In order to prove the first duality result SBB $=\Vc$ in Theorem \ref{thm:main} (i), we focus on the dependence of the value function SBB on its second argument. We thus fix $\mu_0\in\Pc_2(\R^d)$ and we analyse in this section the map $F:\Pc_2(\R^d)\longrightarrow\R$ defined by:
$$
F(m):={\rm SBB}(\mu_0,m),~m\in\Pc_2(\R^d).
$$
Let $\Mc_2$ denote the vector space of finite signed measures with finite second moment,
endowed with the topology $\sigma(\Mc_2,\Cc_w)$.

\begin{lemma}\label{lem:Fconvsci}
The map $F$ is convex and continuous on $\Pc_2(\R^d)$.
\end{lemma}

\proof
{\bf 1.} We first show that $F$ is convex on $\Pc_2(\R^d)$. Let $m_0,m_1\in\Pc_2(\R^d)$ and $\lambda\in(0,1)$. Fix $\eta>0$ and choose
$\P_i\in\Pc(\mu_0,m_i)$ with $\E^{\P_i}\big[\int_0^T c(\gamma_t^{\P_i})\d t \big]
\le F(m_i)+\eta$, $i=0,1$.

Consider the enlarged space $\tilde\Omega:=\Omega\times\{0,1\}$ with canonical process $(\tilde X,U)$ and canonical filtration $\tilde\F=\F^{\tilde X}\vee\F^U$, $\F^{\tilde X}=\{\Fc^{\tilde X}_t=\sigma(\widetilde X_s,s\le t)\}_{t\le T}$ and $\F^{U}=\{\Fc^{U}_t=\sigma(U)\}_{t\le T}$. Let $\tilde\P$ be the probability measure under which $U$ is a Bernoulli$(\lambda)$ variable independent of $\tilde X_0$, $\widetilde X$ has conditional law $\P_i$ given $\{U=i\}$, $i=0,1$, and thus its characteristics $\tilde\gamma$ are defined by $\tilde\gamma^i:=(\tilde\alpha^i,\tilde\sigma^i)$ conditionally on $\{U=i\}$. 

By the Markovian projection argument of Gy\"ongi \cite{gyongy1986mimicking}, we see that $\P:=\tilde\P\circ \widetilde{X}^{-1}\in\Pc(\mu_0,m_\lambda)$, $m_\lambda:=(1-\lambda)m_0+\lambda m_1$, with characteristics $\gamma_t^\P=(\alpha_t^\P,\sigma_t^\P)$ defined by $(\alpha_t^\P,\sigma_t^\P(\sigma_t^\P)^\intercal)=\E^{\widetilde\P}[(\widetilde\alpha_t,\widetilde\sigma_t(\widetilde\sigma_t)^\intercal)| \Fc^{\tilde X}_t]$, and by the convexity of the cost function $c(a,b)$ in $(a,bb^\intercal)$, we see that
\begin{align}
F(m_\lambda)
\le
\E^{\P}\!\int_0^T \!\!c(\gamma_t^{\P})\ dt
& \le 
(1-\lambda) \E^{\P_0}\!\int_0^T \!\! c(\gamma_t^{\P_0})\ \d t
+\lambda\E^{\P_1}\!\int_0^T \!\! c(\gamma_t^{\P_1})\ \d t
\\
&\le 
(1-\lambda) F(m_0)
+\lambda F(m_1)+2\eta.
\end{align}
The required convexity of $F$ now follows from the arbitrariness of $\eta>0$.

\noindent {\bf 2.} We next show the continuity of $F$. We shall prove in Step 3 below that
\begin{equation}\label{ineq:step3}
F(m)
\le
(1+C_0\kappa)F(m')+d(1+\beta)\kappa,
~\mbox{whenever}~
\kappa:=W_2(m,m')\le\frac{T}2.
\end{equation}
Let $m_n\in\Pc_2(\R^d)$ with $m_n\to m$ in $\sigma(\Mc_2,\Cc_w)$.
Then $m_n\rightharpoonup m$ and, since $|.|^2\in\Cc_w$,
$\int|.|^2\d m_n\to\int|.|^2\d m$, which implies $W_2(m_n,m)\to0$. Then, as $\kappa_n:=W_2(m_n,m)\le\frac{T}2$ for large $n$, it follows from \eqref{ineq:step3} that:
$$
F(m)
\le
\liminf_{n\to\infty}(1+C_0\kappa_n)F(m_n)+d(1+\beta)\kappa_n
=\liminf_{n\to\infty}F(m_n).
$$
Applying \eqref{ineq:step3} again with $m$ and $m_n$ exchanged, we obtain the reverse inequality.
Hence
$$
\lim_{n\to\infty}F(m_n)=F(m),
$$
which proves the continuity of $F$.

\noindent {\bf 3.} It remains to prove \eqref{ineq:step3}. Fix $\eta>0$ and pick $\P\in\Pc(\mu_0,m')$ such that
\begin{equation}\label{eq:nearmin}
\E^{\P}\!\int_0^T c(\gamma_t^\P)\d t \le F(m')+\eta.
\end{equation}
\smallskip
\noindent\emph{3-a. Time change on $[0,T-\kappa]$.} Set $\theta:=\frac{T}{T-\kappa}$ and define on the same space  $\widetilde X_u:=X_{\theta u}$ for $u\in[0,T-\kappa]$. Let $\widetilde \P$ be the law of $\widetilde X$ on $C([0,T-\kappa];\R^d)$.
Then $\widetilde X_0\sim\mu_0$ and $\widetilde X_{T-\kappa}=X_T\sim m'$.
A direct change-of-variables computation in the martingale problem shows that
$\widetilde X$ has characteristics $\widetilde\alpha_u=\theta\alpha_{\theta u}^\P$ and $\widetilde\sigma_u=\sqrt\theta\, \sigma_{\theta u}^\P$, and therefore:
$$
\E^{\widetilde \P}\!\int_0^{T-\kappa}
                             c(\widetilde\alpha_t,\widetilde\sigma_t)\ \d t
= 
\frac12\theta \ 
\E^{\P}\!\int_0^T\big(|\alpha_t^\P|^2
                                 +\frac{\beta}{\theta^2} |\sqrt\theta(\sigma_t^\P-I_d)+(\sqrt{\theta}-1)I_d|^2
                          \big)\ \d t.
$$
Note that, for $b\in\R^{d\times d}$, $ |\sqrt{\theta}b+(\sqrt{\theta}-1)I_d|^2 = \theta |b|^2 +2\sqrt{\theta}(\sqrt{\theta}-1)b:I_d +(\sqrt{\theta}-1)^2 d$. By Young's inequality, $2\sqrt{\theta}(\sqrt{\theta}-1)b:I_d \le \theta(\theta-1)|b|^2 + \frac{(\sqrt{\theta}-1)^2}{\theta-1}d$. Hence $|\sqrt{\theta}b+(\sqrt{\theta}-1)I_d|^2 \le \theta^2|b|^2+\frac{\theta(\sqrt{\theta}-1)}{\sqrt{\theta}+1}d$. We deduce from \eqref{eq:nearmin} that
\begin{eqnarray}
\E^{\widetilde \P}\!\int_0^{T-\kappa}
                             c(\widetilde\alpha_t,\widetilde\sigma_t)\ \d t
&\le&
\theta \ 
\E^{\P}\!\int_0^Tc(\alpha_t^\P,\sigma^\P_t)\  \d t
+\frac{T\beta d(\sqrt\theta-1)}{2(\sqrt\theta+1)}
\nonumber\\
&\le&
\theta (F(m')+\eta)
+\frac{T\beta d(\sqrt\theta-1)}{2(\sqrt\theta+1)}.
\label{Flsc1}
\end{eqnarray}

\noindent\emph{3-b. Terminal correction on $[T-\kappa,T]$.}
Let $\pi$ be an optimal $W_2$-coupling between $m'$ and $m$, disintegrated as
$\pi(dy| x)m'(dx)$. Enlarge the space and sample $Y\sim \pi(\cdot\mid \widetilde X_{T-\kappa})$.
Then $\text{\rm Law}(Y)=m$ and $\E|Y-\widetilde X_{T-\kappa}|^2=W_2(m',m)^2=\kappa^2$.
Define
\begin{equation}
\widehat X_t :=
\1_{\{t\in[0,T-\kappa]\}} \widetilde X_t
+\1_{\{t\in[T-\kappa,T]\}}
   \Big\{\widetilde X_{T-\kappa}+\dfrac{t-(T-\kappa)}{\kappa}\bigl(Y-\widetilde X_{T-\kappa}\bigr)\Big\}.
\end{equation}
On $[T-\kappa,T]$, $\widehat X$ has drift $\widehat\alpha_t=(Y-\widetilde X_{T-\kappa})/\kappa$
and diffusion $\widehat\sigma_t\equiv0$, hence
\begin{equation}\label{eq:bridge}
\E\!\int_{T-\kappa}^T c(\widehat\alpha_t,\widehat\sigma_t)\ dt
=\frac{W_2(m',m)^2}{2\kappa} + \frac{\beta}{2} d\ \kappa
=\frac{\kappa}{2} + \frac{\beta}{2} d\ \kappa.
\end{equation}
Let $\widehat \P$ be the law of $\widehat X$ on the canonical space.
By a Markovian projection argument as in Step 1 of the current proof, we see that $\widehat \P\in\Pc(\mu_0,m)$ and
\begin{align}\label{eq:jensenbridge}
F(m)
\le
\E^{\widehat \P}\!\int_0^T c(\alpha_t^{\widehat \P},\sigma_t^{\widehat \P})\ dt
&\le \E\!\int_0^T c(\widehat\alpha_t,\widehat\sigma_t)\ dt
\\
&\le
\theta (F(m')+\eta)
+\frac{T\beta d(\sqrt\theta-1)}{2(\sqrt\theta+1)}
+\frac{\kappa}{2} + \frac{\beta}{2} d\ \kappa,
\end{align}
where the last inequality follows from \eqref{Flsc1} and \eqref{eq:bridge}. Substituting $\theta:=\frac{T}{T-\kappa}$ and using $\kappa\le T/2$, we have $\theta=\frac{T}{T-\kappa}\le 1+\frac{2\kappa}{T}, \  \frac{\sqrt\theta-1}{\sqrt\theta+1}\le C_T\kappa$  for some constant $C_T<\infty$ depending only on $T$. Letting $\eta\downarrow0$ and enlarging the constant if necessary gives \eqref{ineq:step3}.
\ep

\vspace{5mm}
We are now ready for the first duality result in Theorem \ref{thm:main} {\rm (i)}.

\begin{proposition}\label{prop:1stduality}
For $\mu_0,\mu_T\in\Pc_2(\R^d)$, we have ${\rm SBB}(\mu_0,\mu_T)=\Vc(\mu_0,\mu_T)$.
\end{proposition}

\proof  {\bf 1.} Define $\bar F:\Mc_2\to \R\cup\{+\infty\}$ by
$$
\bar F(m):=
\left\{
\begin{array}{ll}
F(m), & \mbox{if } m\in\Pc_2(\R^d),\\
+\infty, & \mbox{if } m\in\Mc_2\setminus\Pc_2(\R^d).
\end{array}
\right.
$$
Since $F$ is convex on $\Pc_2(\R^d)$, the map $\bar F$ is convex on $\Mc_2$.
Moreover, $\bar F$ is $\sigma(\Mc_2,\Cc_w)-$lower semicontinuous on $\Mc_2$. Then,
$F(\mu_T)=\bar F(\mu_T)=\bar F^{**}(\mu_T)$ by the Fenchel-Moreau theorem, where:
$$
\bar F^{**}(m)=\sup_{\psi\in\Cc_w}m(\psi)-\bar F^*(\psi)
~\mbox{and}~
\bar F^{*}(\psi)=\sup_{m\in\Mc_2}m(\psi)-\bar F(m)
=\sup_{m\in\Pc_2}m(\psi)-F(m)
$$
Hence $
F(\mu_T)=\bar F(\mu_T)=\bar F^{**}(\mu_T).
$

\medskip
\noindent {\bf 2.} We now complete the required dual representation by identifying the map $\bar F^{**}$ to the value $\Vc(\mu_0,\mu_T)$ introduced in \eqref{weakduality}. First,
\begin{align}
\bar F^*(\psi)
&= \sup_{m\in\Pc_2}m(\psi)-F(m)
= \sup_{m\in\Pc_2}m(\psi)-\inf_{\P\in\Pc(\mu_0,m)}J^0(\P)
= \sup_{\P\in\Pc(\mu_0)}-J^\psi(\P),
\end{align}
and by Step 1,
$$
F(\mu_T)
=\bar F^{**}(\mu_T)
=
\sup_{\psi\in\Cc_w}\mu_T(\psi)-\bar F^*(\psi)
=
\sup_{\psi\in\Cc_w}\mu_T(\psi)-\sup_{\P\in\Pc(\mu_0)}-J^\psi(\P).
$$
We now prove that $\bar F^{**}(\mu_T)=\Vc(\mu_0,\mu_T)$ as defined in \eqref{weakduality}. This is similar to Lemma 3.5 in \cite{TanTouzi2013} which was established under their coercivity condition on the cost function.

{\it 2-a.} We first show that it suffices to prove the required result for potential maps $\psi$ with $|\psi^+|_{\infty}<\infty$. Indeed, for $\psi\in\Cc_w$, define $\psi_k:=\psi\wedge k$, and notice that as $\psi_k\uparrow\psi$, $c\ge0$, and $\psi^-(X_T)\in \L^1(\P)$, it follows from the monotone convergence theorem that
$J^{\psi_k}(\P)\downarrow J^\psi(\P)$. By the standard approximate optimiser, we also obtain the monotone convergence $\inf_{\P\in\Pc(\mu_0)}J^{\psi_k}(\P)
\downarrow\inf_{\P\in\Pc(\mu_0)}J^{\psi}(\P)$.

\noindent {\it 2-b.} The inequality $\inf_{\P\in\Pc(\mu_0)}J^\psi(\P)\ge-\mu_0(V_0^\psi)$ was already established right before introducing the dual problem $\Vc(\mu_0,\mu_T)$ in \eqref{weakduality}. We now prove the reverse inequality by means of a measurable selection argument.
Since $|\psi^+|_\infty<\infty$ and $c\ge0$, we have
$-J^\psi(\P)
= \E^\P\Big[\psi(X_T)-\int_0^T c(\gamma_t^\P)\d t\Big]
\le |\psi^+|_\infty, \  \P\in\Pc(\mu_0)$, and thus $V^\psi_0<\infty$. For fixed $\eta>0$, it follows
from the measurable selection result, see e.g. \cite[Appendix, Thm.~A.1]{TanTouzi2013}, that there exists a $\mu_0$-measurable map $x\mapsto\P^{x,\eta}\in\Pc(\delta_x)$ such that $
J^{\psi}(\P^{x,\eta}) \le -V^\psi_0(x)+\eta
~\text{for }\mu_0\text{-a.e. }x.$ Define $\P^\eta:=\int_{\R^d} \P^{x,\eta}(\cdot)\,\mu_0(dx)$ on $\Omega$. As $\P^\eta[\cdot\,|\,X_0=x]=\P^{x,\eta}$, $\mu_0$-a.s., we have $\P^\eta\in\Pc(\mu_0)$, and it follows from Fubini's theorem that
\begin{equation}
J^{\psi}(\P^\eta)
=\int_{\R^d} J^\psi(\P^{x,\eta})\,\mu_0(dx)
\le -\int_{\R^d} V^\psi_0(x)\,\mu_0(dx)+\eta.
\end{equation}
This shows that $\P^\eta$ is an $\eta-$optimiser, and we deduce the required result by the arbitrariness of $\eta>0$.
\ep

\vspace{5mm}

In preparation for the proof of the reduced dual formulation in the next subsection, we now show that we may restrict the dual maximization to upper bounded $\beta$--concave potentials $\psi$: 
\begin{equation}
\Cc_{w,\uparrow}^{\conc}:=\Cc^{\rm conc}\cap\Cc_{w,\uparrow}
~\mbox{with}~
\Cc_{w,\uparrow}:=\{\psi\in \Cc_w:\ |\psi^+|_\infty<\infty\}.
\end{equation}
We also set
$\Cc_{w,\uparrow}^{\conv}
:= \big\{\mathbf T_\beta^-[\psi]:\psi\in\Cc_{w,\uparrow}^{\conc}\big\}$.

For $t<T$ and $x\in\R^d$, we use the shifted notation
$\Pc_t(\delta_x)$ for the admissible laws on $[t,T]$ starting from $x$ at time $t$, and
\begin{equation}
V_t^\psi(x)
:= \sup_{\P\in\Pc_t(\delta_x)}
\E^\P\Big[ \psi(X_T)-\int_t^T c(\gamma_r^\P)\,\d r \Big].
\end{equation}

\begin{lemma}\label{lem:reduce0}
We have
\begin{equation}
\Vc(\mu_0,\mu_T)
= \underline{\Vc}(\mu_0,\mu_T)
:= \sup_{\psi\in\Cc_{w,\uparrow}^{\conc}}
\{\mu_T(\psi)-\mu_0(V_0^\psi)\}.
\end{equation}
Moreover,
\begin{equation}\label{eq:dual-enlarged-conc}
\mu_T(\psi)-\mu_0(V_0^\psi)
\le \Vc(\mu_0,\mu_T), ~~\mbox{for all}~~ \psi\in\Cc^{\conc},
\end{equation}
where the objective is understood in the extended sense.
\end{lemma}

\proof
(a) We first restrict the dual maximization to $\Cc_{w,\uparrow}$.
For arbitrary $\psi\in\Cc_w$, we have $\psi^n:=\psi\wedge n\in\Cc_{w,\uparrow}$, $n\ge1$. As $\psi^n\uparrow\psi$ and $\psi^n\ge\psi^1\in\L^1(\mu_T)$, it follows by monotone convergence that $
\mu_T(\psi^n)\ \uparrow\ \mu_T(\psi).$

On the other hand, monotonicity in the terminal payoff gives $V_0^{\psi^n}\uparrow V_0^\psi$ pointwise. As $A:=|\frac\psi{w}|_\infty<\infty$, and the Wiener measure $\Q^{0,x}\in\Pc(\delta_x)$ started from $\delta_x$ has characteristics $\gamma^{\Q^{0,x}}=(0,I_d)$ inducing the cost $c(\gamma^{\Q^{0,x}})=0$, we have
\begin{equation}
V_0^{\psi^n}(x)\ge E^{\Q^{0,x}}[\psi^n(X_T)]
\ge -A(1+\E^{\Q^{0,x}}|X_T|^2)
= -A(1+|x|^2+dT),
~\mbox{for all}~n\ge 1.
\end{equation}
Then $V_0^{\psi^n}\ge V_0^{\psi^1}\in \L^1(\mu_0)$, and we get by monotone convergence $\mu_0(V_0^{\psi^n})\ \uparrow\ \mu_0(V_0^\psi).$ 

Hence $\mu_T(\psi^n)-\mu_0(V_0^{\psi^n})\longrightarrow \mu_T(\psi)-\mu_0(V_0^\psi)$, and we deduce from the arbitrariness of $\psi\in\Cc_w$ that $\Vc(\mu_0,\mu_T)
=
\sup_{\psi\in\Cc_{w,\uparrow}}
\{\mu_T(\psi)-\mu_0(V_0^\psi)\}$.

\smallskip\noindent
(b) We next show that we may restrict further the dual maximization to $\beta$--concave potentials.
Fix now $\psi\in \Cc_{w,\uparrow}$ and define its $\beta$--concave envelope
\begin{equation}
\hat\psi:=\Tc_\beta^+\circ \Tc_\beta^-[\psi]\in \Cc_w^{\conc},
~\text{so that}~\psi\le \hat\psi\le |\psi^+|_\infty.
\end{equation}
On the other hand, notice that $V^\psi$ is l.s.c. as the supremum of continuous maps, and it follows from the easy part of dynamic programming principle (DPP) that $V^\psi$ is a l.s.c.\ viscosity
supersolution of the HJB equation \eqref{HJB} on $[0,T)\times\R^d$. Since the Hamiltonian $H$ is infinite outside its domain ${\rm dom}(H):=\{(p,A):A<\beta I_d\}$, we deduce that $V_t^\psi$ is $\beta$--concave for every $t<T$.
As $V_t^\psi\ge \psi-\frac{\beta d}2(T-t)$ (take the control $\gamma\equiv0$), we have
\begin{equation}\label{eq:terminal-layer-beta-envelope}
\liminf_{\substack{s\uparrow T\\ y\to x}} V_s^\psi(y)\ge \hat\psi(x),
\qquad x\in\R^d .
\end{equation}
We now propagate this terminal estimate backward by dynamic programming. Fix
$(t,x)\in[0,T)\times\R^d$ and let $\P\in\Pc_t(\delta_x)$ be such that
$\E^\P\int_t^T c(\gamma_r^\P)\,\d r<\infty$.
For $s\in(t,T)$, the dynamic programming principle gives
$V_t^\psi(x)
\ge \E^\P\left[ V_s^\psi(X_s)-\int_t^s c(\gamma_r^\P)\,\d r \right]$.
Let $s\uparrow T$. Since $X_s\to X_T$, the boundary layer estimate
\eqref{eq:terminal-layer-beta-envelope} gives
$\liminf_{s\uparrow T} V_s^\psi(X_s)\ge \hat\psi(X_T)$, $\P$-a.s.
To apply Fatou's lemma, we use the zero control only as a lower bound. Starting from
$z$ at time $s$, the control $(\alpha,\sigma)=(0,0)$ gives
$V_s^\psi(z)
\ge \psi(z)-\frac{\beta d}{2}(T-s)$.
Since $\psi\in\Cc_w$, there exists $C>0$ such that
$\psi(z)\ge -C(1+|z|^2)$.
Hence, for $s$ close to $T$,
\begin{equation}
V_s^\psi(X_s)
\ge
-C\Big(1+\sup_{r\in[t,T]}|X_r|^2\Big)-\frac{\beta d}{2}T.
\end{equation}
The right-hand side is $\P$-integrable by the same BDG estimate as in
Lemma~\ref{lem:supX2}. Fatou's lemma therefore yields
$\liminf_{s\uparrow T}\E^\P[V_s^\psi(X_s)]
\ge \E^\P[\hat\psi(X_T)]$.
Moreover,
\begin{equation}
\int_t^s c(\gamma_r^\P)\,\d r
\uparrow
\int_t^T c(\gamma_r^\P)\,\d r.
\quad
\text{Consequently,}
\quad
V_t^\psi(x)
\ge
\E^\P\left[
\hat\psi(X_T)-\int_t^T c(\gamma_r^\P)\,\d r
\right].
\end{equation}
Taking the supremum over all such $\P\in\Pc_t(\delta_x)$ gives
$V_t^\psi(x)\ge V_t^{\hat\psi}(x)$.
Together with the opposite inequality $V_t^\psi\le V_t^{\hat\psi}$, we obtain
$V_t^\psi=V_t^{\hat\psi}
\quad\mbox{on }[0,T)\times\R^d$.
In particular $V_0^\psi=V_0^{\hat\psi}$, and therefore
$\mu_T(\psi)-\mu_0(V_0^\psi)
\le \mu_T(\hat\psi)-\mu_0(V_0^{\hat\psi})$,
which yields the desired restriction to $\beta$--concave potentials.

\medskip
\noindent (c)
Fix $\psi\in \Cc^{\rm conc}$ and $(t,x)\in[0,T)\times\R^d$.
Fix any $y\in\R^d$ and consider on $[t,T]$ the deterministic admissible characteristics
$\alpha_s\equiv \frac{y-x}{T-t},
\ 
\sigma_s\equiv 0$.
Then $X_T=y$ a.s.\ and
$\int_t^T c(\alpha_s,\sigma_s)\ ds
=
\frac{|y-x|^2}{2(T-t)}+\frac{\beta d}{2}(T-t)$,
since $| I_d|^2=\tr( I_d)=d$.
Hence
\begin{equation}\label{eq:V-lower-stepc}
V_t^\psi(x)
\ge
\psi(y)-\frac{|y-x|^2}{2(T-t)}-\frac{\beta d}{2}(T-t)
\ge
-C_{t,\psi}(1+|x|^2).
\end{equation}

For $n\ge 1$ set $\psi_n:=\psi\wedge n$. Then $\psi_n\uparrow\psi$ and $\psi_n(y)\ge \psi_1(y)$, so \eqref{eq:V-lower-stepc} yields the uniform-in-$n$ bound
$$
V_t^{\psi_n}(x)\ge \psi_1(y)-\frac{|y-x|^2}{2(T-t)}-\frac{\beta d}{2}(T-t)\ge -C_t(1+|x|^2),
\qquad n\ge 1,
$$
with $C_t<\infty$ depending on $t,\beta,d,y,\psi_1(y)$ but not on $n$.

\medskip
\noindent (d)
Fix $\psi\in\Cc^{\rm conc}$ and set $\psi_n:=\psi\wedge n$.

\smallskip
\noindent\emph{(d1) $\mu_T(\psi_n)\uparrow \mu_T(\psi)$ in the extended sense.}
Since $\psi_n=\psi\wedge n$, we have $(\psi_n)^-=\psi^-$ and $(\psi_n)^+\uparrow \psi^+$.
Because $\psi(x)\le C(1+|x|^2)$ and $\mu_T\in\Pc_2(\R^d)$, we have $\psi^+\in \L^1(\mu_T)$.
If $\int_{\R^d}\psi^-\ d\mu_T=\infty$, then
$
\int_{\R^d}\psi_n\ d\mu_T
=
\int_{\R^d}(\psi_n)^+\ d\mu_T-\int_{\R^d}\psi^-\ d\mu_T
=
-\infty
\ \mbox{for all }n,
$
so $\mu_T(\psi_n)\uparrow \mu_T(\psi)=-\infty$. If instead $\int_{\R^d}\psi^-\ d\mu_T<\infty$, then $\psi^-\in \L^1(\mu_T)$, and we obtain $\mu_T(\psi_n)\uparrow\mu_T(\psi)$ by monotone convergence.

\smallskip
\noindent\emph{(d2) Weak duality on the enlarged class.}
Fix $\P\in\Pc(\mu_0,\mu_T)$. If $\E^\P\int_0^T c(\gamma_s^\P)\ ds=+\infty$, then there is nothing to prove. We may thus assume that
$\E^\P\int_0^T c(\gamma_s^\P)\ ds<\infty$.
Let $\{\P_x\}_{x\in\R^d}$ be a regular conditional law of $\P$ given $X_0$.
Then $\P_x\in\Pc(\delta_x)$ for $\mu_0$--a.e.\ $x$, and by definition
$V_0^{\psi_n}(x)
\ge \E^{\P_x}\!\left[ \psi_n(X_T)-\int_0^T c(\gamma_s^{\P_x})\,\d s \right]
= -J^{\psi_n}(\P_x)$.
Integrating with respect to $\mu_0$ and using disintegration,
$
\mu_T(\psi_n)-\mu_0(V_0^{\psi_n})
\le
\E^\P\int_0^T c(\gamma_s^\P)\ ds.
$
Taking the infimum over $\P\in\Pc(\mu_0,\mu_T)$ and using Proposition~\ref{prop:1stduality}, we obtain \begin{equation}\label{eq:weak-dual-n-inf-stepd} \mu_T(\psi_n)-\mu_0(V_0^{\psi_n}) \le {\rm SBB}(\mu_0,\mu_T) = \Vc(\mu_0,\mu_T), \qquad n\ge 1. \end{equation}

\smallskip
\noindent\emph{(d3) Sending $n\to\infty$.}
Since $\psi_n\uparrow\psi$, monotonicity in the terminal payoff gives $V_0^{\psi_n}\uparrow V_0^\psi$ pointwise, hence
$V_0^{\psi_n}\le V_0^\psi$.
Moreover, by Step (c) with $t=0$, there exists $C<\infty$ such that
$
V_0^{\psi_n}(x)\ge -C(1+|x|^2), \  n\ge 1,
$
so that $\mu_0(V_0^{\psi_n})$ and $\mu_0(V_0^\psi)$ are well-defined in $(-\infty,+\infty]$ and
$\mu_0(V_0^{\psi_n})\le \mu_0(V_0^\psi)$.
Thus from \eqref{eq:weak-dual-n-inf-stepd},
$
\mu_T(\psi_n)-\mu_0(V_0^\psi)
\le
\mu_T(\psi_n)-\mu_0(V_0^{\psi_n})
\le
\Vc(\mu_0,\mu_T).
$
Letting $n\to\infty$ and using $\mu_T(\psi_n)\uparrow \mu_T(\psi)$ from (d1), we obtain
\begin{equation}\label{eq:weak-dual-extended-stepd}
\mu_T(\psi)-\mu_0(V_0^\psi)\le \Vc(\mu_0,\mu_T),
\qquad \psi\in\Cc^{\rm conc}.
\end{equation}
This proves \eqref{eq:dual-enlarged-conc}.
\ep

\section{Reduced dual verification} \label{sec:dualreduced} 

The first duality result ${\rm SBB}=\Vc$ was established in the last section, and we now complement it by justifying equality with the reduced dual value $\Vc_{\rm red}$ as claimed in Theorem \ref{thm:main} (i). We start with the analysis of the regularity of the maps $u^\phi$.

\begin{lemma}\label{lem:HC_Moreau}
Fix $T,\beta>0$ and $\psi\in\bar\Cc^{\conc}$. Set $\phi:=\mathbf{T}_\beta^-[\psi]$. Then:
\begin{enumerate}
\item[{\rm (1)}] $h^\phi_t:=\Nc_{T-t}*e^\phi >0$ for all $t\le T$, and $h^\phi\in C^{1,3}(Q_T)$ satisfies
\begin{equation}
\partial_t h^\phi+\tfrac12\Delta h^\phi=0~\text{on } Q_T.
\end{equation}
\item[{\rm (2)}] $\tilde u^\phi:=\log{h^\phi}\!\in\! C^{1,3}(Q_T)$, with $D^2\tilde u^\phi_t\!+\!\kappa(t)I\succeq 0,$ $\kappa(t)\!:=\!\frac{\beta}{1+\beta(T-t)}$, $t\!<\!T$, and
\begin{equation}
\partial_t \tilde u^\phi+\tfrac12\big(\Delta \tilde u^\phi+|\nabla \tilde u^\phi|^2\big)=0~\text{on } Q_T.
\end{equation}
\item[{\rm (3)}] $v^\phi:=\mathbf{T}_\beta^+[\tilde u^\phi]\in C^{1,2}(Q_T)$, and the minimum in $\mathbf{T}_\beta^+[\tilde u^\phi_t](x)$ is uniquely attained at $\msY_t(x)$, for all $(t,x)\in Q_T$. Moreover,
\\
{\rm (3-a)} $\nabla v^\phi=\beta({\rm id}-\msY)$, $\partial_t v^\phi=\partial_t \tilde u^\phi(\msY)$, and $D^2 v^\phi-\beta I=-\beta^2\big(\beta I+D^2\tilde u^\phi(\msY)\big)^{-1}$. 
In particular, $-\frac1{T-t}I\preceq D^2 v^\phi_t\prec \beta I$, for all $t<T$.
\\ 
{\rm (3-b)} $v^\phi$ is a solution of the HJB equation \eqref{HJB}.
\\
{\rm (3-c)} If, in addition, $\psi\in\Cc_{w,\uparrow}^{\conc}$, then $v^\phi_t\to \psi$ locally uniformly as $t\uparrow T$, and $v^\phi$ has quadratic growth uniformly in $t$, i.e. there exists a constant $C<\infty$ such that
\begin{equation}
|v^\phi_t(x)|\le C(1+|x|^2),
\qquad (t,x)\in Q_T.
\end{equation}
\end{enumerate}
\end{lemma}

\proof
By definition of $\bar\Cc^{\conc}$, we have
$\mu_0\Big(\mathbf{T}_\beta^+\Big[u_T^{\mathbf{T}_\beta^-[\psi]}\Big]
\Big)<\infty$.
Since $\phi=\mathbf{T}_\beta^-[\psi]$, it follows that
$\mu_0\big(\mathbf{T}_\beta^+[u_T^\phi]\big)<\infty$.
Hence $\mathbf{T}_\beta^+[u_T^\phi](x_*)<\infty$ for some $x_*\in\R^d$, and by definition of $\mathbf{T}_\beta^+$, we deduce that $u_T^\phi(y_0)<\infty$ for some $y_0\in\R^d$.
Equivalently, $(\Nc_T*e^\phi)(y_0)<\infty$. We continue the proof in several steps.

\medskip
\noindent{\bf 1.} We first show that $h^\phi$ and $\tilde u^\phi$ are finite on $Q_T$. 

Fix $0<t<T$ and $x\in\R^d$. Then
$\frac{\Nc_{T-t}(x-z)}{\Nc_T(y_0-z)}
=
\big(\frac T{T-t}\big)^{\frac d2}
e^{q_x(z)}$ with $q_x(z):=
-\frac{|x-z|^2}{2(T-t)}+\frac{|y_0-z|^2}{2T}$.
Notice that
$q_x(z)=-\big(\frac1{2(T-t)}-\frac1{2T}\big)|z|^2
+\big(\frac{x}{T-t}-\frac{y_0}{T}\big)\!\cdot z
-\frac{|x|^2}{2(T-t)}+\frac{|y_0|^2}{2T}$ is a concave quadratic polynomial in $z$ since $t<T$. Hence it is bounded above on $\R^d$, so there exists
$C_{t,x}<\infty$ such that
$\Nc_{T-t}(x-z)\le C_{t,x}\,\Nc_T(y_0-z),
\ z\in\R^d$.
Multiplying by $e^\phi(z)$ and integrating, we obtain
\begin{equation}\label{eq:positive-time-finite}
0<\Nc_{T-t}*e^\phi(x)\le C_{t,x}(\Nc_T*e^\phi)(y_0)<\infty
~\mbox{for all}~(t,x)\in Q_T.
\end{equation}

\medskip
\noindent{\bf 2.}
Fix a compact strip
$K=[t_0,\tau]\times B_R\subset (0,T)\times\R^d,
\  0<t_0<\tau<T$.
Choose $s_*\in(T-t_0,T)$. By \eqref{eq:positive-time-finite},
$(\Nc_{s_*}*e^\phi)(y_0)<\infty$.
For every multi-index $\alpha$ and every $m\in\{0,1\}$ with $|\alpha|+2m\le3$, we have
$\partial_s^m D_x^\alpha \Nc_s(x-z)=P_{m,\alpha,s}(x-z)\,\Nc_s(x-z)$,
where $P_{m,\alpha,s}$ is a polynomial of degree at most $3$, and its coefficients are bounded for
$s\in[T-\tau,T-t_0]$.

Since $T-t\in[T-\tau,T-t_0]$ on $K$, the same Gaussian comparison as in the previous Step~{\bf 1} gives
$\sup_{(t,x)\in K}\big|\partial_s^m D_x^\alpha \Nc_{T-t}(x-z)\big|
\le
C_{K,m,\alpha}\,\Nc_{s_*}(y_0-z)$.
Because
$\int_{\R^d}\Nc_{s_*}(y_0-z)e^\phi(z)\,dz<\infty$,
differentiation under the integral sign is justified on $K$. Therefore
$h^\phi\in C^{1,3}(Q_T)$,
\
$\partial_t h^\phi+\frac12\Delta h^\phi=0$.
Since $h^\phi>0$, also
$\tilde u^\phi=\log h^\phi\in C^{1,3}(Q_T)$,
\
$\partial_t\tilde u^\phi+\frac12\big(\Delta\tilde u^\phi+|\nabla\tilde u^\phi|^2\big)=0$.

\medskip
\noindent{\bf 3.}
Fix $t\in(0,T)$ and set
$s:=T-t,
\
\kappa(t):=\frac{\beta}{1+\beta s}$.
Let
$G(y):=\phi(y)+\frac{\beta}{2}|y|^2$.
Since $\phi$ is $\beta$--convex, $G$ is convex.
We compute
$\tilde u^\phi_t(x)
=
\log\int_{\R^d}(2\pi s)^{-\frac d2}e^{\phi(y)-\frac{|x-y|^2}{2s}}\d y$.
A direct completion of the square gives
$-\frac{\beta}{2}|y|^2-\frac{|x-y|^2}{2s}+\frac{\kappa(t)}2|x|^2
=
-\frac{1+\beta s}{2s}\big|y-\frac{x}{1+\beta s}\big|^2$.
Therefore
\begin{align}
\tilde u^\phi_t(x)+\frac{\kappa(t)}2|x|^2
&= -\frac d2\log(2\pi s)
+ \log\int_{\R^d}
e^{G(y)-\frac{1+\beta s}{2s}|y-\frac{x}{1+\beta s}|^2}\d y.
\\
&= -\frac d2\log(2\pi s)
+ \log\int_{\R^d} e^{G(z+\frac{x}{1+\beta s})-\frac{1+\beta s}{2s}|z|^2}\d z,
\end{align}
by the change of variables $z=y-\frac{x}{1+\beta s}$. For each fixed $z$, the map
$x\longmapsto G\big(z+\frac{x}{1+\beta s}\big)-\frac{1+\beta s}{2s}|z|^2$
is convex. By H\"older's inequality, the logarithm of the integral of its exponential is convex. Hence
$x\longmapsto \tilde u^\phi(t,x)+\frac{\kappa(t)}2|x|^2$
is convex, and therefore
\begin{equation}\label{eq:semiconvex-u}
D^2\tilde u^\phi_t+\kappa(t)I\succeq0.
\end{equation}

\noindent{\bf 4.}
For $(t,x,y)\in(0,T)\times\R^d\times\R^d$, define
$F(t,x,y):=\tilde u^\phi_t(y)+\frac{\beta}{2}|x-y|^2$.
By \eqref{eq:semiconvex-u},
$D_y^2F(t,x,y)=D^2\tilde u^\phi_t(y)+\beta I
\succeq (\beta-\kappa(t))I$.
Since
$\beta-\kappa(t)=\frac{\beta^2(T-t)}{1+\beta(T-t)}>0$,
the map $y\mapsto F(t,x,y)$ is strongly convex for each fixed $(t,x)$.

To see it attains its minimum, fix $(t,x)$. By Taylor's formula at $y=0$,
\begin{equation}
F(t,x,y)
\ge
F(t,x,0)+\nabla_yF(t,x,0)\cdot y+\frac{\beta-\kappa(t)}{2}|y|^2.
\end{equation}
Hence
$F(t,x,y)\ge \frac{\beta-\kappa(t)}{2}|y|^2-C_{t,x}|y|-C_{t,x}$,
which tends to $+\infty$ as $|y|\to\infty$. Thus $F(t,x,\cdot)$ is coercive, so it has a unique minimizer $\msY_t(x)$, and
$v^\phi_t(x):=F\bigl(t,x,\msY_t(x)\bigr)$.

We next prove that $\msY$ is continuous on $Q_T$. Let $(t_n,x_n)\to(t,x)$ and set $y_n:=\msY(t_n,x_n)$.
For all large $n$, we have $t_n\in[t/2,(t+T)/2]$. On this compact time interval,
$c_*:=\inf_{r\in[t/2,(t+T)/2]}(\beta-\kappa(r))>0$.
Using strong convexity at $y=0$,
\begin{equation}
F(t_n,x_n,y)\ge F(t_n,x_n,0)+\nabla_yF(t_n,x_n,0)\cdot y+\frac{c_*}{2}|y|^2.
\end{equation}
As $(t_n,x_n)$ stays in a compact set, both $F(t_n,x_n,0)$ and $\nabla_yF(t_n,x_n,0)$ are bounded uniformly in $n$,
so
$F(t_n,x_n,y)\ge \frac{c_*}{2}|y|^2-C|y|-C$.
Because $y_n$ minimizes $F(t_n,x_n,\cdot)$,
$F(t_n,x_n,y_n)\le F(t_n,x_n,0)$,
and the right-hand side is uniformly bounded. Hence $(y_n)$ is bounded. Passing to a convergent subsequence if needed,
$y_n\to y_*$, and continuity of $F$ gives
$F(t,x,y_*)=\lim_n F(t_n,x_n,y_n)\le \lim_n F(t_n,x_n,y)=F(t,x,y)$, for all $y\in\R^d$.
Thus $y_*$ minimizes $F(t,x,\cdot)$, and by uniqueness $y_*=\msY_t(x)$. Hence $\msY$ is continuous.

For fixed $t$, the minimizer satisfies the first-order condition
\begin{equation}\label{eq:first-order-Y}
\nabla_yF(t,x,\msY(t,x))=G_t(x,\msY_t(x))=0,
~\mbox{with}~
G_t(x,y):=\nabla\tilde u^\phi_t(y)-\beta(x-y).
\end{equation}
Notice that $G_t$ is $C^1$ with
$D_yG_t(x,y)=D^2\tilde u^\phi_t(y)+\beta I_d$,
which is invertible by \eqref{eq:semiconvex-u}. Then $\msY_t$ is $C^1$ by the implicit functions theorem. Differentiating \eqref{eq:first-order-Y} in $x$, we obtain
\begin{equation}\label{eq:DxY}
\nabla\msY_t(x)=\beta\bigl(D^2\tilde u^\phi_t(\msY_t(x))+\beta I_d\bigr)^{-1}.
\end{equation}
We now derive the envelope identities. Since
$v^\phi_t(x)=\tilde u^\phi_t(\msY_t(x))+\frac{\beta}{2}|x-\msY_t(x)|^2$,
the chain rule and \eqref{eq:first-order-Y}-\eqref{eq:DxY} provide
\begin{align}
\nabla v^\phi_t(x)&=\beta(x-\msY(t,x))=\nabla\tilde u^\phi_t(\msY_t(x))
\label{eq:grad-v}
\\
D^2v^\phi_t(x)
&=
\beta I_d-\beta^2\bigl(D^2\tilde u^\phi_t(\msY(t,x))+\beta I_d\bigr)^{-1},
\label{eq:hess-v}
\end{align}
and $D^2v^\phi$ inherits the continuity of $\msY$.

It remains to compute $\partial_t v^\phi$. We do not differentiate $\msY$ in time. Fix $(t,x)$ and let $h>0$.
Since $\msY_t(x)$ is admissible in the minimization defining $v^\phi_{t+h}(x)$,
$v^\phi_{t+h}(x)\le F\bigl(t+h,x,\msY_t(x)\bigr)$,
hence
\begin{equation}
\frac{v^\phi_{t+h}(x)-v^\phi_t(x)}{h}
\le
\frac{F(t+h,x,\msY_t(x))-F(t,x,\msY_t(x))}{h}
=
\frac{\tilde u^\phi_{t+h}(\msY_t(x))-\tilde u^\phi_t(\msY_t(x))}{h},
\end{equation}
and then $\limsup_{h\downarrow0}\frac{v^\phi_{t+h}(x)-v^\phi_t(x)}{h}
\le
\partial_t\tilde u^\phi_t(\msY_t(x))$.
For the reverse inequality, since $\msY_{t+h}(x)$ is admissible in the minimization defining $v^\phi_t(x)$,
$v^\phi_t(x)\le F\bigl(t,x,\msY_{t+h}(x)\bigr)$,
so
\begin{equation}
\frac{v^\phi_{t+h}(x)\!-\!v^\phi_t(x)}{h}
\ge
\frac{F(t\!+\!h,x,\msY_{t+h}(x))\!-\!F(t,x,\msY_{t+h}(x))}{h}
=
\frac{\tilde u^\phi_{t+h}(\msY_{t+h}(x))\!-\!\tilde u^\phi_t(\msY_{t+h}(x))}{h}.
\end{equation}
Because $\msY$ is continuous and $\partial_t\tilde u^\phi$ is continuous,
$\msY_{t+h}(x)\to \msY_t(x)$ and $\partial_t\tilde u^\phi_t(\msY_{t+h}(x))\to \partial_t\tilde u^\phi_t(\msY_t(x))$. Therefore
$\liminf_{h\downarrow0}\frac{v^\phi_{t+h}(x)-v^\phi_t(x)}{h}
\ge
\partial_t\tilde u^\phi_t(\msY_t(x))$.
Hence
\begin{equation}\label{eq:dt-v}
\partial_t v^\phi_t(x)=\partial_t\tilde u^\phi_t(\msY_t(x)).
\end{equation}
Since $\msY$ and $\partial_t\tilde u^\phi$ are continuous, $\partial_t v^\phi$ is continuous. Thus
$v^\phi\in C^{1,2}(Q_T)$.

\medskip
\noindent{\bf 5.}
Let $\lambda$ be an eigenvalue of $D^2\tilde u^\phi_t(\msY_t(x))$, and let $\ell$ be the corresponding eigenvalue of $D^2v^\phi_t(x)$.
By \eqref{eq:hess-v},
$\ell
=
\beta-\frac{\beta^2}{\beta+\lambda}
=
\frac{\beta\lambda}{\beta+\lambda}$.
Since \eqref{eq:semiconvex-u} gives $\lambda\ge -\kappa(t)$, and $\lambda\mapsto \beta\lambda/(\beta+\lambda)$ is increasing on
$(-\beta,\infty)$,
$-\frac{\beta\kappa(t)}{\beta-\kappa(t)}\le \ell<\beta$.
Using
$\frac{\beta\kappa(t)}{\beta-\kappa(t)}=\frac{1}{T-t}$,
we obtain
$-\frac{1}{T-t}I\preceq D^2v^\phi(t,\cdot)\prec \beta I_d$.

Recall that the Hamiltonian for $A<\beta I_d$,
$H(p,A)
=
\frac12|p|^2+\frac{\beta}{2}(\tr(I-\frac1\beta A)^{-1}-d)$.
With $p=\nabla v^\phi_t(x)$ and $A=D^2v^\phi_t(x)$, and using \eqref{eq:hess-v},
\begin{equation}
\Big(I-\frac1\beta D^2v^\phi_t(x)\Big)^{-1}
=
I+\frac1\beta D^2\tilde u^\phi_t(\msY_t(x)).
\end{equation}
Therefore
$H(\nabla v^\phi,D^2v^\phi)(t,x)
=
\frac12|\nabla\tilde u^\phi_t(\msY_t(x))|^2+\frac12\Delta\tilde u^\phi_t(\msY_t(x))$.
Using \eqref{eq:grad-v}, \eqref{eq:dt-v}, and the Cole--Hopf equation for $\tilde u^\phi$, we obtain
$\partial_t v^\phi+H(\nabla v^\phi,D^2v^\phi)=0
\ \text{on }Q_T$.

\smallskip
\noindent {\bf 6.} For $\psi\in\Cc_{w,\uparrow}^{\conc}$, we have $-Cw\le\psi\le C$ for some constant $C$, implying that $\phi\le C$ and
$\phi(y)
\ge \sup_{x\in\R^d}\left\{-C(1+|x|^2)-\frac{\beta}{2}|x-y|^2\right\}
\ge -C'(1+|y|^2)$ for some constant $C'<\infty$.
As $\phi$ is bounded from above, we deduce that $h^\phi_t\to e^\phi$ locally uniformly as $t\uparrow T$. The local uniform convergence $\tilde u^\phi(t,\cdot)\to \phi$ is inherited from that of $h^\phi_t$ towards $e^\phi$.
By the stability of the Moreau envelop under local uniform convergence, we see that
$v^\phi_t=\mathbf{T}_\beta^+[\tilde u^\phi_t]\to \mathbf{T}_\beta^+[\phi]=\psi$ locally uniformly.

We next justify the claimed growth for $v^\phi$. First, $v^\phi_t\le\tilde u^\phi_t\le \sup\phi$.
Moreover, by the $\beta-$convexity of $\phi$, for each $y$ there exists $\xi\in\partial(\phi+\frac{\beta}{2}|\cdot|^2)(y)$ such that
\begin{equation}
\phi(z)\ge \phi(y)+(\xi-\beta y)\cdot(z-y)-\frac{\beta}{2}|z-y|^2,
~\mbox{for all}~z\in\R^d.
\end{equation}
Using this inequality with $s=T-t$ and $z=y+W_s$, $W_s\sim \Nc(0,sI)$, we get
\begin{align}
h^\phi_t(y)
=\E\big[e^{\phi(y+W_s)}\big]
&\ge e^{\phi(y)}\ \E\big[e^{(\xi-\beta y)\cdot W_s-\frac{\beta}{2}|W_s|^2}\big]
\\
&=e^{\phi(y)}\ (1+\beta s)^{-d/2}e^{\frac{s}{2(1+\beta s)}|\xi-\beta y|^2}
\ \ge\ e^{\phi(y)}\ (1+\beta T)^{-d/2}.
\end{align}
Then
$\tilde u^\phi_t(y)=\log h^\phi_t(y)
\ge \phi(y)-\frac{d}{2}\log(1+\beta T)=:\phi(y)-C_T$, and
\begin{equation}
v^\phi_t(x)=\inf_y\Big\{\tilde u^\phi_t(y)+\frac{\beta}{2}|x-y|^2\Big\}
\ge \inf_y\Big\{\phi(y)+\frac{\beta}{2}|x-y|^2\Big\}-C_T
=\psi(x)-C_T.
\end{equation}
Since $\psi\in\Cc_{w,\uparrow}^{\conc}$, this proves the claimed quadratic growth of $v^\phi$.
\ep

\vspace{5mm}

We now have all ingredients to identify the maps $V^{\psi}$ and $v^\phi$ by means of a verification argument. We first note that
\begin{equation}\label{eq:Cwuparrow-in-barC}
\Cc_{w,\uparrow}^{\conc}\subset \bar\Cc^{\conc}.
\end{equation}

\begin{lemma}\label{lem:verification}
Let $\psi\in\Cc^{\rm conc}_{w,\uparrow}$ and set $\phi:=\mathbf{T}_\beta^-[\psi]$. Then $V^{\psi}=v^\phi$ on $Q_T$.
\end{lemma}

\proof In the rest of this proof, we denote $t_m:=T-\frac1m$, $m\in\N$. 

\smallskip
\noindent {\bf 1.} Let $\P\in\Pc_t(\delta_x)$ be such that $\E^\P\int_t^T c(\gamma_s^\P)\,\d s<\infty$, and recall from Lemma~\ref{lem:supX2} that $\E^\P[\sup_{r\in[t,T]}|X_r|^2]<\infty$. By the $C^{1,2}$ regularity of $v:=v^\phi$ and its quadratic growth established in Lemma \ref {lem:HC_Moreau} (3), and since the stochastic integral in It\^o's formula is a true martingale, we obtain
$$
\E^\P[v_{t_m}(\!X_{t_m})]-v_t(x)
\!=\!
\E^\P\!\Big[\int_t^{t_m}\!\!\!\!\big(\partial_sv_s+\alpha_s^\P\cdot\nabla v_s+\frac12\,\sigma_s^\P(\sigma_s^\P)^\intercal\!\!:\!\!D^2v_s\big)(X_s)\d s\Big]
\le
\E^\P\!\Big[\int_t^{t_m}\!\!\!\!\!c(\gamma_s^\P)\d s\Big],
$$
where the last inequality follows from the fact that $v$ is a supersolution of the HJB equation \eqref{HJB}. Since $X_{t_m}\to X_T$ $\P$-a.s., $v_{t_m}\to\psi$ locally uniformly, as $m\nearrow\infty$, and
$|v_{t_m}(X_{t_m})|\le C(1+\sup_{r\in[t,T]}|X_r|^2)\in\L^1$, we see that $v_{t_m}(X_{t_m})\to\psi(X_T)$ in $\L^1(\P)$ by dominated convergence. Then, as the cost function $c\ge0$, we may now deduce by monotone convergence that
$$
v_t(x)
\ge
\E^\P\Big[\psi(X_T)-\int_t^T c(\gamma_s^\P)\,\d s\Big],
$$
and by arbitrariness of $\P\in\Pc_t(\delta_x)$, we get $v_t(x)\ge V_t^\psi(x)$.

\smallskip
\noindent {\bf 2.} We next prove the reverse inequality $v_t(x)\le V_t^\psi(x)$. By Lemma~\ref{lem:HC_Moreau} (3),
the minimizer $\msY_t(x)$ in $v_t(x)=\Tc_\beta^+[\tilde u_t](x)$ is unique. Set $y:=\msY_t(x)$, so that $x=\msX_t(y)$. Let $Y$ be a $\Q$-Brownian motion on $[t,t_m]$ with $Y_t=y$, and notice that the upper boundedness of $\phi$ implies that $Z_s:=\frac{h_s^\phi(Y_s)}{h_t^\phi(y)}$ is a positive bounded martingale inducing an equivalent probability measure $\Q^m$ on $\Fc_T$ by $\d \Q^m:=Z_{t_m}\,\d \Q$. By Girsanov's theorem the process
$W_s:=Y_s-\int_t^{s\wedge t_m}\nabla\log h_r^\phi(Y_r)\,\d r$
defines a $\Q^m-$Brownian motion on $[0,T]$, and we may then write the dynamics of $Y$ as
\begin{equation}\label{eq:Y_sde}
\d Y_s=\nabla\log h_s^\phi(Y_s)\1_{s\in [t,t_m]}\,\d s+\d W_s,\qquad Y_t=y.
\end{equation}
Denote
$X_s:=\msX_s(Y_s)\1_{s\in [t,t_m]}+(X_{t_m}+\int_{t_m}^s \d W_s)\1_{s\in (t_m,T]}$,
where we recall that $\msX_s:=id+\frac1\beta\nabla\tilde u_s$. By
Lemma~\ref{lem:HC_Moreau} (3-a), it follows that
\begin{equation}\label{eq:grad-v-XY}
\nabla v_s(X_s)=\beta(X_s-Y_s)=\nabla\tilde u_s(Y_s)=\nabla\log h^\phi(s,Y_s),
\qquad s\in[t,t_m].
\end{equation}
Set
$\widehat\alpha_s:=\nabla v_s,
\
\widehat\sigma_s:=\Big( I_d-\frac{D^2v_s}{\beta}\Big)^{-1}$.
Since $\msY_s:=id-\frac1\beta\nabla v_s$ is the inverse of $\msX_s$, differentiating the identity
$\msY_s(\msX_s(y))=y$ yields
$D\msX_s(y)=\Big( I_d-\frac{D^2v_s(\msX_s(y))}{\beta}\Big)^{-1}$,
and therefore
$D\msX_s(Y_s)=\widehat\sigma_s(X_s),
\ s\in[t,t_m]$.
By Lemma~\ref{lem:HC_Moreau} {\it (1)--(2)}, we have $\tilde u^\phi\in C^{1,3}(Q_T)$, and therefore
$(s,y)\longmapsto \msX_s(y):=y+\frac1\beta\nabla\tilde u^\phi_s(y)$ is of class $C^{1,2}$ on $Q_T$.
Applying Itô to $X_s=X(s,Y_s)$ on $[t,t_m]$ using \eqref{eq:grad-v-XY}, we obtain
\begin{align}
\d X_s
&=
\Big(
\partial_s\msX_s+D\msX_s\,\nabla\tilde u_s+\frac12\Delta\msX_s
\Big)(Y_s)\,\d s
+
D\msX_s(Y_s)\,\d W_s
\\
&=
\Big(
\nabla\tilde u_s+\frac1\beta\Big(
\nabla\partial_s\tilde u_s+D^2\tilde u_s\,\nabla\tilde u_s+\frac12\nabla\Delta\tilde u_s
\Big)
\Big)(Y_s)\,\d s
+
\widehat\sigma_s(X_s)\,\d W_s.
\end{align}
Since $h^\phi$ solves the backward heat equation, differentiating the PDE of $\tilde u=\log h^\phi$ gives
$\nabla\partial_s\tilde u_s+D^2\tilde u_s\,\nabla\tilde u_s+\frac12\nabla\Delta\tilde u_s=0$.
Hence, using also \eqref{eq:grad-v-XY}, we get
\begin{align}
X_t=x:=\msX_t(y),~
&dX_s=\widehat\alpha_s(X_s)\,ds+\widehat\sigma_s(X_s)\,dW_s,~ s\in[t,t_m],
\\
&\mbox{and}~
X_s=X_{t_m}+W_s-W_{t_m}~\mbox{for}~s\in[t_m,T].
\end{align}
Let $\P^{t,x,m}$ be the law of $X$
under $\Q^m$. Then $\P_{t,x}^{m}\in\Pc_t(\delta_x)$ with characteristics $\gamma^{\P_{t,x}^{m}}_s=(\widehat\alpha,\widehat\sigma)(s,X_s)\1_{s\in[t,t_m]}+(0,I_d)\1_{s\in[t_m,T]}$. As $c(\gamma^{\P_{t,x}^{m}})=0$ on $[t_m,T]$, it follows from the HJB equation \eqref{HJB} satisfied by $v$ that
\begin{equation}\label{eq:exact_trunc}
v_t(x)
=
\E^{\P_{t,x}^{m}}\Big[
v_{t_m}(X_{t_m})-\int_t^{t_m} c(\gamma^{\P_{t,x}^{m}}_s)\,\d s
\Big].
\end{equation}
Let $\delta_m:=\sqrt{T-t_m}$, let $Z$ be a standard Gaussian random variable independent of
$\Fc_{t_m}$, and define
$\psi_m(x):=\E[\psi(x+\delta_m Z)]$.
Since $c(\gamma^{\P_{t,x}^{m}})=0$ on $[t_m,T]$, we have
$$\E^{\P_{t,x}^{m}}\Big[ \psi(X_T)-\int_t^T c(\gamma^{\P_{t,x}^{m}}_s)\,\d s \Big]
= \E^{\P_{t,x}^{m}}\Big[
\psi_m(X_{t_m})-\int_t^{t_m} c(\gamma^{\P_{t,x}^{m}}_s)\,\d s
\Big].$$
Consequently,
$v_t(x)
= \E^{\P_{t,x}^{m}}\Big[ \psi(X_T)-\int_t^T c(\gamma^{\P_{t,x}^{m}}_s)\,\d s \Big] +\eps_m
\le V_t^\psi(x)+\eps_m$,
where $\eps_m := \E^{\P_{t,x}^{m}}\big[ v_{t_m}(X_{t_m})-\psi_m(X_{t_m}) \big]$.  Since $(v_{t_m},\psi_m)\to(\psi,\psi)$ locally uniformly as $m\nearrow\infty$ and both
$v_{t_m}$ and $|\psi_m|$ have quadratic growth uniformly in $m$, it remains to justify the
uniform square-integrability of $(X_{t_m})_{m\ge1}$ under $(\P_{t,x}^{m})_{m\ge1}$. 

We first record a uniform linear-growth bound. Since
$\psi\in\Cc^{\rm conc}_{w,\uparrow}$, the function
$\phi=\mathbf T_\beta^-[\psi]$ is bounded from above and satisfies a quadratic lower bound. Hence, by the same convexity argument as in Lemma~\ref{lem:HC_Moreau}, there exists a constant $C_t<\infty$ such that
$|\nabla u_r^\phi(z)|\le C_t(1+|z|),
\  0<r\le T-t,\ z\in\R^d $.
Equivalently,
\[
|\nabla \tilde u_s^\phi(z)|\le C_t(1+|z|),
\qquad t\le s<T,\ z\in\R^d .
\]
Under $\Q^m$, the process $Y$ satisfies
$\d Y_s=\nabla\tilde u_s^\phi(Y_s)\,\d s+\d W_s,\  s\in[t,t_m]$,
and therefore the preceding linear-growth estimate, together with the BDG inequality and Gronwall's lemma, gives
$\sup_{m\ge1}\E^{\Q^m}\Big[\sup_{t\le s\le t_m}|Y_s|^2\Big]<\infty$.
Since
$X_{t_m}
= Y_{t_m}+\frac1\beta\nabla\tilde u_{t_m}^\phi(Y_{t_m})$,
the same linear-growth bound yields
$\sup_{m\ge1}\E^{\P_{t,x}^{m}}\big[|X_{t_m}|^2\big]<\infty$.
Thus, for every $R>0$,
\[
\sup_{|z|\le R}|v_{t_m}(z)-\psi_m(z)|\longrightarrow0,
\]
while the tails are controlled uniformly by the quadratic growth and the above uniform
square-integrability. Hence $\eps_m\to0$, inducing the required inequality
$v_t(x)\le V_t^\psi(x)$.  
\ep

\vspace{5mm}

We also record that the relaxed dual has the same value. Indeed, by \eqref{eq:Cwuparrow-in-barC}, we have
$\Cc_{w,\uparrow}^{\conc}\subset \bar\Cc^{\conc}$. Hence
\begin{equation}
\bar\Vc(\mu_0,\mu_T)
\ge \sup_{\psi\in\Cc_{w,\uparrow}^{\conc}}
\{\mu_T(\psi)-\mu_0(V_0^\psi)\}
= \Vc(\mu_0,\mu_T).
\end{equation}
On the other hand, since $\bar\Cc^{\conc}\subset\Cc^{\conc}$, Lemma~\ref{lem:reduce0} gives
$\bar\Vc(\mu_0,\mu_T)\le \Vc(\mu_0,\mu_T)$.
Therefore
\begin{equation}\label{eq:barV-equals-V}
\bar\Vc(\mu_0,\mu_T)=\Vc(\mu_0,\mu_T).
\end{equation}

\noindent {\bf Proof of Theorem \ref{thm:main} (i)} 
By Moreau biconjugation
$\{\phi=\mathbf{T}_\beta^-[\psi]:\ \mbox{for } \psi\in\Cc^{\rm conc}_{w,\uparrow}\}
= \Cc^{\rm conv}_{w,\uparrow}$.
Combining Proposition~\ref{prop:1stduality}, Lemma~\ref{lem:reduce0},
Lemma~\ref{lem:verification}, and \eqref{eq:barV-equals-V}, we obtain
\begin{equation}\label{dualuparrow}
{\rm SBB}(\mu_0,\mu_T)
= \Vc(\mu_0,\mu_T)
= \bar\Vc(\mu_0,\mu_T)
= \sup_{\phi\in\Cc^{\rm conv}_{w,\uparrow}}
\left\{
\mu_T(\mathbf{T}_\beta^+[\phi])-\mu_0(v_0^\phi)
\right\}.
\end{equation}
It remains to prove that the value of the supremum on the right-hand side is not affected by enlarging the set of dual potential maps from $\Cc^{\rm conv}_{w,\uparrow}$ to $\Cc^{\rm conv}_w$. To see this, fix $\phi\in\Cc^{\rm conv}_w$ and set
$\psi:=\mathbf T_\beta^+[\phi]$. For $n\ge1$, define
\[
\chi_n:=\psi\wedge n,
\qquad
\phi_n:=\mathbf T_\beta^-[\chi_n].
\]
Since $\chi_n\in\Cc_w$ is bounded from above, we have
$\phi_n\in\Cc^{\rm conv}_{w,\uparrow}$. Moreover,
$\mathbf T_\beta^+[\phi_n]=\mathbf T_\beta^+\mathbf T_\beta^-[\chi_n]$
is the $\beta$--concave envelope of $\chi_n$. Hence
$\chi_n
\le \mathbf T_\beta^+[\phi_n]
\le \psi$,
where the last inequality follows from the fact that $\psi$ is itself a $\beta$--concave majorant of $\chi_n$. Since $\chi_n\uparrow\psi$ and $\psi\in\Cc_w$, we obtain
$\mu_T\big(\mathbf T_\beta^+[\phi_n]\big)
\longrightarrow
\mu_T(\psi)$.
On the other hand, $\chi_n\le\psi$ implies
$\phi_n
= \mathbf T_\beta^-[\chi_n]
\le \mathbf T_\beta^-[\psi]
= \phi$.
Therefore
$u_T^{\phi_n}\le u_T^\phi,
\ \mathbf T_\beta^+[u_T^{\phi_n}]
\le \mathbf T_\beta^+[u_T^\phi]$,
and consequently
$\mathfrak J(\phi_n)
= \mu_T\big(\mathbf T_\beta^+[\phi_n]\big)
- \mu_0\big(\mathbf T_\beta^+[u_T^{\phi_n}]\big)
\ge \mu_T\big(\mathbf T_\beta^+[\phi_n]\big)
- \mu_0\big(\mathbf T_\beta^+[u_T^\phi]\big)$.
Letting $n\to\infty$ yields
$\liminf_{n\to\infty}\mathfrak J(\phi_n)
\ge \mu_T(\psi) - \mu_0\big(\mathbf T_\beta^+[u_T^\phi]\big)
= \mathfrak J(\phi)$.
This proves that the supremum over $\Cc^{\rm conv}_{w,\uparrow}$ is equal to the supremum over $\Cc^{\rm conv}_w$.
\ep

\section{Dual attainment}
\label{sec:dualattainment}

Our objective in this section is to prove Theorem \ref{thm:main} (ii-a), namely the existence of potential $\hat\psi$ attaining the dual value $\Vc(\mu_0,\mu_T)$, and a dual potential $\hat\phi$ attaining the reduced dual value $\Vc_{\rm red}(\mu_0,\mu_T)$. We shall denote throughout:
\begin{equation}\label{Udeltaphi}
U^\phi_\eta
:=
u_{\eta T}^{\phi}(0)
=
\log\E[e^{\phi(W_{\eta T})}],
~\text{for some fixed}~\eta\in(0,1).
 \end{equation}

\begin{lemma}\label{lem:bounds}
For every $\phi\in\Cc^{\rm conv}_{w,\uparrow}$ with $\phi(0)=0$, we have:
\begin{enumerate}
\item[\rm (i)] There is a constant $\nu_\eta$ (depending on $\beta,T,d,\eta$) such that
\begin{equation}
-\nu_\eta\le U^\phi_\eta
\le
-\frac d2\log(\eta)+\inf_{x\in\R^d} u^\phi_T(x)+\frac{|x|^2}{2(1-\eta)T}.
\end{equation}

\item[\rm (ii)] There exists a constant $\nu'_\eta$ (depending on $\beta,T,d,\eta$), such that for all $x\in\R^d$:
\begin{subequations}
\begin{align}
\hspace{-15mm}u_s^\phi(x)
& \ge
    -\frac{\beta|x|^2}{2}
    +\1_{\{s>0\}}\Big(\frac{\beta^2|x|^2}{4s}-\frac{(\nu'_\eta\!+\!U^\phi_\eta)^2}{\beta^2s}\Big)
    \!-\!\1_{\{s=0\}}(\nu'_\eta\!+\!U^\phi_\eta)|x|,
~s\!\in\![0,T],\label{eq:u_low_bound}
\\
\hspace{-15mm} u^\phi_s(x)
&\le
U^\phi_\eta+\nu_\eta
-\frac d2\log\Big(1-\frac{s}{\eta T}\Big)
+\frac{|x|^2}{2(\eta T-s)},
~s\in[0,\eta T).
\label{eq:u_up_bound}
\end{align}
\end{subequations}
\end{enumerate}
\end{lemma}

\begin{proof}
(i) Using the Young inequality $|x-y|^2\le\frac{|x|^2}{1-\eta}+\frac{|y|^2}{\eta}$, we see that
\begin{align}
e^{u_T^\phi(x)}
=
\int \frac{e^{\phi(y)-\frac1{2T}|x-y|^2}}{(2\pi T)^{\frac d2}}\,\d y
\ge
\eta^{\frac d2}e^{-\frac{|x|^2}{2(1-\eta)T}}
\int \frac{e^{\phi(y)-\frac{|y|^2}{2\eta T}}}{(2\pi \eta T)^{\frac d2}}\,\d y
=
\eta^{\frac d2}e^{-\frac{|x|^2}{2(1-\eta)T}}e^{U^\phi_\eta},
\end{align}
which provides the right-hand side inequality in (i).

Next, by the $\beta$--convexity of $\phi$, we obtain the Jensen inequality for all $R>0$:
\begin{equation}
\phi(x)
\le
-\frac12\beta|x|^2
+\frac1{|B_R|}\int_{B_R(x)}\Big(\phi+\frac12\beta|.|^2\Big)(y)\,\d y
\le
\frac12\beta R^2+\frac1{|B_R|}\int_{B_R(x)}\phi(y)\,\d y.
\end{equation}
Applying Jensen's inequality again to the exponential function, we deduce that
\begin{align}
\phi(x)
&\le
\frac12\beta R^2
+\log\Big(\frac1{|B_R|}\int_{B_R(x)}e^{\phi(y)}\,\d y\Big)
\nonumber\\
&\le
\frac12\beta R^2
+\log\Big(\frac1{|B_R|}e^{\frac{|x|^2+R^2}{2\eta T}}
\int_{B_R(x)}e^{\phi(y)-\frac{|y|^2}{2\eta T}}\,\d y\Big)
\nonumber\\
&\le
\frac12\beta R^2
+\log\Big(\frac{(2\pi\eta T)^{\frac d2}}{|B_R|}
e^{\frac{|x|^2+R^2}{2\eta T}}e^{U^\phi_\eta}\Big)
\nonumber\\
&=
U^\phi_\eta+\frac{|x|^2}{2\eta T}+\bar\nu_\eta(R),
~\text{with}~
\bar\nu_\eta(R)
:=
\frac12\Big(\beta+\frac1{\eta T}\Big)R^2
+\log\frac{(2\pi\eta T)^{\frac d2}}{|B_R|}.
\label{phileq1}
\end{align}
As $\phi(0)=0$, this provides the left-hand side inequality in (i) with the finite constant
\begin{equation}
\nu_\eta:=\inf_{R>0}\bar\nu_\eta(R).
\end{equation}

\medskip
(ii) Notice that \eqref{phileq1} provides the required upper bound in
\eqref{eq:u_up_bound} at $s=0$. To obtain the lower bound in \eqref{eq:u_low_bound}
at $s=0$, note that the $\beta$--convexity of $\phi$ implies that
\begin{equation}\label{eq:phi_lb}
\phi(x)+\frac12\beta|x|^2\ge p\cdot x
~\text{for all }p\in\partial\phi(0),
\end{equation}
as $\phi(0)=0$. Therefore $|p|=p\cdot\frac{p}{|p|}
\le
\frac{\beta}{2}+\max_{B_1}\phi
\le
\frac{\beta}{2}+\frac1{2\eta T}+\nu_\eta+U^\phi_\eta
=:\nu'_\eta+U^\phi_\eta,$
by substituting $R=1$ in \eqref{phileq1}. Then $\phi(x)+\frac12\beta|x|^2\ge p\cdot x\ge-(\nu'_\eta+U^\phi_\eta)|x|,$
which is exactly the case $s=0$ in \eqref{eq:u_low_bound}, since $u^\phi_0=\phi$.

By \eqref{eq:phi_lb}, the definition of $u_s^\phi$, and a direct Gaussian calculation,
\begin{align}
e^{u_s^\phi(x)}
=
\int_{\R^d}\frac{e^{\phi(y)-\frac{|x-y|^2}{2s}}}{(2\pi s)^{\frac d2}}\,\d y
\ge
\int_{\R^d}\frac{e^{p\cdot y-\frac{\beta}{2}|y|^2-\frac{|x-y|^2}{2s}}}{(2\pi s)^{\frac d2}}\,\d y
=
\frac{e^{\frac{|\frac{x}{s}+p|^2}{2(\beta+\frac1s)}-\frac{|x|^2}{2s}}}{(1+\beta s)^{\frac d2}}.
\end{align}
Using the Young inequality $\frac{p\cdot x}{1+\beta s}
\ge
-\frac{\lambda}{2}|x|^2-\frac{1}{2\lambda}\frac{|p|^2}{(1+\beta s)^2},$
$\lambda:=\frac{\beta^2 s}{2(1+\beta s)}>0,$ this implies
\begin{equation}
u_s^\phi(x)
\ge
-\frac d2\log(1+\beta s)
-\frac{\beta}{4}(2-\beta s)|x|^2
-\frac{|p|^2}{(1+\beta s)\beta^2s},
\end{equation}
which gives the case $s>0$ in \eqref{eq:u_low_bound}, after increasing
$\nu'_\eta$ if necessary so that $\nu'_\eta+U^\phi_\eta\ge0$, by using
$|p|^2\le(\nu'_\eta+U^\phi_\eta)^2$.
We finally derive \eqref{eq:u_up_bound} by using the right-hand side of \eqref{phileq1}:
\begin{align}
u^\phi_s(x)
=
\log \E\big[e^{\phi(x+W_s)}\big]
&\le
U^\phi_\eta+\nu_\eta
+\log \E\Big[e^{\frac{1}{2\eta T}|x+W_s|^2}\Big]
\\
&=
U^\phi_\eta+\nu_\eta
-\frac d2\log\Big(1-\frac{s}{\eta T}\Big)
+\frac{|x|^2}{2(\eta T-s)},
\qquad s\in[0,\eta T). 
\end{align}
\end{proof}

\medskip
{\it In the rest of this paper,
\begin{equation}\label{assum:existence}
\mbox{we assume}~
\delta:=1-\frac1{\beta T}>0
~\mbox{and we fix}~
\eta\in(0,\delta).
\end{equation}}
\noindent {\bf Proof of Theorem \ref{thm:main} (ii-a)}
By \eqref{dualuparrow}, we may consider a maximizing sequence $(\phi_n)_{n\ge 1}$ for Problem~\eqref{D} satisfying:
\begin{equation}\label{maxsequence}
\phi_n\in\Cc^{\rm conv}_{w,\uparrow},
\phi_n(0)=0,
~\mbox{and}~
\mathfrak{J}(\phi_n)
:=
\mu_T\big(\mathbf{T}^+_\beta[\phi_n]\big)
-
\mu_0\big(\mathbf{T}^+_\beta[u^{\phi_n}_T]\big)
\stackrel{n\to\infty}{\longrightarrow}
\Vc_{\rm red}.
\end{equation}
As $\mathfrak{J}(\phi+{\rm const})=\mathfrak{J}(\phi)$, we may set $\phi_n(0)=0$ for all $n\ge 1$. We proceed in several steps.

\medskip
\noindent {\bf 1.} We first prove that the sequence
\begin{equation}\label{Udeltabdd}
\big(U^{\phi_n}_\eta\big)_{n\ge 1}
\mbox{ is bounded}.
\end{equation}
First, $U^{\phi_n}_\eta\ge-\nu_\eta$ by Lemma~\ref{lem:bounds}~(i). To derive a uniform upper bound, we observe that $\mathbf{T}_\beta^+[\phi_n]\le \phi_n(0)+\frac{\beta}{2}|.|^2=\frac{\beta}{2}|.|^2,$ so that
Lemma~\ref{lem:bounds}~(i) yields $
u_T^{\phi_n}(y)\ge U_\eta^{\phi_n}-\frac{|y|^2}{2(1-\eta)T}+\frac d2\log(\eta).
$ Then
\begin{equation}
\mathbf{T}_\beta^+\big[u_T^{\phi_n}\big](x)
\ge
U_\eta^{\phi_n}+\frac d2\log(\eta)
+
\mathbf{T}_\beta^+\Big[-\frac{|.|^2}{2(1-\eta)T}\Big](x).
\end{equation}
Since $\eta<\delta$, we have $\frac1{(1-\eta)T}<\beta$, and therefore
$\mathbf{T}_\beta^+\Big[-\frac{|.|^2}{2(1-\eta)T}\Big](x)
=
-\frac{\beta}{2\big(\beta T(1-\eta)-1\big)}\,|x|^2.
$
It follows that, for all sufficiently large $n$,
\begin{align}
\Vc_{\rm red}-1
\le
J(\phi_n)
&=
\int \mathbf{T}_\beta^+[\phi_n](x)\,\mu_T(\d x)
-
\int \mathbf{T}_\beta^+\big[u_T^{\phi_n}\big](x)\,\mu_0(\d x)
\\
&\le
\frac{\beta}{2}\int |x|^2\,\mu_T(\d x)
-
U_\eta^{\phi_n}
-
\frac d2\log(\eta)
+
\frac{\beta}{2\big(\beta T(1-\eta)-1\big)}
\int |x|^2\,\mu_0(\d x).
\end{align}
Since $\mu_0,\mu_T$ have finite second moments and $\Vc_{\rm red}<\infty$, this proves \eqref{Udeltabdd}.

\medskip
\noindent {\bf 2.} We next show that, after passing to a subsequence $(n_k)_k$, we may find a limit function:
\begin{equation}\label{phinconv}
\phi_{n_k}
\longrightarrow
\hat\phi \in \Cc^{\rm conv}_w
\quad\mbox{as $k\to\infty$, locally uniformly on } \R^d.
\end{equation}
Indeed, inequalities \eqref{eq:u_low_bound}--\eqref{eq:u_up_bound} of Lemma~\ref{lem:bounds}~(ii), evaluated at $s=0$,
together with \eqref{Udeltabdd}, show that the sequence $(\phi_n)_n$ is locally bounded.
As each $\phi_n$ is $\beta$--convex and locally uniformly bounded, it follows that $(\phi_n)_n$ is locally Lipschitz,
uniformly in $n$; see Rockafellar \cite{Rockafellar}, Theorem~10.6, p.~88.
The convergence in \eqref{phinconv} is then a direct consequence of the Arzel\`a--Ascoli theorem. Notice that $\hat\phi\in\Cc^{\rm conv}_w$ as it inherits the $\beta-$convexity of the maps $\phi_{n_k},$ together with the quadratic bounds in \eqref{eq:u_low_bound}--\eqref{eq:u_up_bound}. 

\medskip
\noindent {\bf 3.} In this step, we provide the proof of existence of a dual potential optimiser for the reduced dual problem $\Vc_{\rm red}$. We rely on the following claim which will be justified later in Steps {\bf 5} to {\bf 7} below:
\begin{equation}\label{Tbeta+phinconv}
\mathbf{T}_\beta^+[\phi_{n_k}]
\longrightarrow
\mathbf{T}_\beta^+[\hat\phi]
\quad\mbox{as $k\to\infty$, locally uniformly on } \R^d.
\end{equation}
\begin{itemize}
\item[(3a)] Since $\phi_n(0)=0$, we have $\mathbf{T}_\beta^+[\phi_n]\le \frac{\beta}{2}|.|^2\in\L^1(\mu_T)$ by Assumption~\ref{assum:existence}. It then follows from \eqref{Tbeta+phinconv} and Fatou's lemma that
\begin{equation}
\limsup_{k\to\infty}\mu_T\big( \mathbf{T}_\beta^+[\phi_{n_k}]\big)
\le
\mu_T\big(\mathbf{T}_\beta^+[\hat\phi]\big).
\end{equation}

\item[(3b)] Using Lemma~\ref{lem:bounds}~(ii), \eqref{eq:u_low_bound}, with $s=T$, and the boundedness of
$\big(U_\eta^{\phi_n}\big)_{n\ge 1}$ proved in Step {\bf 1} that $
u_T^{\phi_{n_k}}+\frac{\beta}{2}|.|^2\ge -c_0+\frac{\beta^2}{4T}|.|^2$ for some constant $c_0$. Then, 
\begin{align}
\mathbf{T}_\beta^+\big[u^{\phi_{n_k}}_T\big](x)
&=
\inf_y\Big\{u^{\phi_{n_k}}_T(y)+\frac{\beta}{2}|x-y|^2\Big\}
\\
&\ge
-c_0+\frac{\beta}{2}|x|^2+\inf_y\Big\{\frac{\beta^2}{4T}|y|^2-\beta x\cdot y\Big\}
=
-c_0+\Big(\frac{\beta}{2}-T\Big)|x|^2,
\end{align}
for all $x\in\R^d$, where the last lower bound is in $\L^1(\mu_0)$ as $\mu_0\in\Pc_2(\R^d)$. We may then apply Fatou's lemma to get
\begin{equation}\label{lsc1}
\liminf_{k\to\infty}\mu_0\big(\mathbf{T}_\beta^+\big[u^{\phi_{n_k}}_T\big]\big)
\ge
\mu_0\big(L\big),
~~
L:=\liminf_{k\to\infty}\mathbf{T}_\beta^+\big[u^{\phi_{n_k}}_T\big].
\end{equation}
After passing to a subsequence $n_{k'}=n_{k'}^x$ for fixed $x\in\R^d$, we have
\begin{align}
L(x)
=
\lim_{k'\to\infty}\mathbf{T}_\beta^+\big[u^{\phi_{n_{k'}}}_T\big](x)
&\ge
\limsup_{k'\to\infty}
\Big(
\frac{\beta}{2}|x|^2
-\beta x\cdot y_{n_{k'}}
-c_0
+\frac{\beta^2}{4T}|y_{n_{k'}}|^2
\Big),
\end{align}
with $y_{n_{k'}}$ a minimizer of $\mathbf{T}_\beta^+\big[u^{\phi_{n_{k'}}}_T\big](x)$. Then if $L(x)<\infty$, this shows that the sequence $(y_{n_{k'}})_{k'}$ is bounded; after possibly passing to a further subsequence, we may therefore assume that $y_{n_{k'}}\to\hat y\in\R^d$. It follows that
\begin{align}
L(x)
&=
\lim_{k'\to\infty}
\Big(
u^{\phi_{n_{k'}}}_T(y_{n_{k'}})+\frac{\beta}{2}|x-y_{n_{k'}}|^2
\Big)
\\
&\ge
\liminf_{k'\to\infty}u^{\phi_{n_{k'}}}_T(y_{n_{k'}})
+\frac{\beta}{2}|x-\hat y|^2
\\
&=
\log\Big(
\liminf_{k'\to\infty}
\E\big[e^{\phi_{n_{k'}}(y_{n_{k'}}+W_T)}\big]
\Big)
+\frac{\beta}{2}|x-\hat y|^2
\\
&\ge
u^{\hat\phi}_T(\hat y)+\frac{\beta}{2}|x-\hat y|^2
\ge
\inf_{y\in\R^d}
\Big\{
u^{\hat\phi}_T(y)+\frac{\beta}{2}|x-y|^2
\Big\}
=
\mathbf{T}_\beta^+\big[u^{\hat\phi}_T\big](x),
\end{align}
by Fatou's lemma and the local uniform convergence of $\phi_{n_{k'}}$ to $\hat\phi$, where $u_T^{\hat\phi}(y)\in(-\infty,\infty]$ is understood in the extended sense. The same inequality $
L(x)\ge \mathbf{T}_\beta^+\big[u^{\hat\phi}_T\big](x)
$ is trivially true if $L(x)=\infty$. Plugging this into \eqref{lsc1}, we obtain
\begin{equation}\label{lsc}
\liminf_{k\to\infty}
\mu_0\big(\mathbf{T}_\beta^+\big[u^{\phi_{n_k}}_T\big]\big)
\ge
\mu_0\big(\mathbf{T}_\beta^+\big[u^{\hat\phi}_T\big]\big).
\end{equation}
\item[(3c)] Set $
A_k:=\int \mathbf{T}_\beta^+[\phi_{n_k}]\,\d\mu_T$ and $B_k:=\int \mathbf{T}_\beta^+\big[u_T^{\phi_{n_k}}\big]\,\d\mu_0,$ $k\ge 1$.
By {\rm (3a)}, the sequence $(A_k)_k$ is bounded above, and by \eqref{maxsequence}, we see that $A_k-B_k=\mathfrak{J}(\phi_{n_k})\longrightarrow \Vc_{\rm red}$. Hence $(B_k)_k$ is bounded above. Together with \eqref{lsc}, this implies
\begin{equation}\label{hatpsiintegrability}
\mu_0\big(\mathbf{T}_\beta^+\big[u_T^{\hat\phi}\big]\big)<\infty,
\end{equation}
so that $\mathfrak{J}(\hat\phi)$ is well defined. Combining the conclusion from {\rm (3a)} with \eqref{lsc}, we now obtain
\begin{align}
\Vc_{\rm red}
=
\lim_{k\to\infty}\mathfrak{J}(\phi_{n_k})
&\le
\limsup_{k\to\infty}\int \mathbf{T}_\beta^+[\phi_{n_k}]\,\d\mu_T
-
\liminf_{k\to\infty}\int \mathbf{T}_\beta^+\big[u_T^{\phi_{n_k}}\big]\,\d\mu_0
\\
&\le
\int \mathbf{T}_\beta^+[\hat\phi]\,\d\mu_T
-
\int \mathbf{T}_\beta^+\big[u_T^{\hat\phi}\big]\,\d\mu_0
=
\mathfrak{J}(\hat\phi),
\end{align}
which shows that $\Vc_{\rm red}=\mathfrak{J}(\hat\phi)$, as required. 
\end{itemize}

\medskip
\noindent{\bf 4.} We now prove that the induced potential $\hat\psi:=\Tc_\beta^+[\hat\phi]$ attains the supremum in the relaxed dual problem $\bar\Vc$.

\smallskip
\noindent 4(a).
Fix $(t,x)\in[0,T)\times\R^d$ and let $\P\in \Pc_t(\delta_x)$ satisfy
$\E^\P\int_t^T c(\gamma_r^\P)\ \d r<\infty$. 
Since $\P\circ X_t^{-1}=\delta_x$, we have $X_t=x$, $\P$-a.s. Hence, after subtracting at time
$t$ and defining
\begin{equation}
W_s:=W_s^\P-W_t^\P,\qquad \alpha_s:=\alpha_s^\P,\qquad \sigma_s:=\sigma_s^\P,\qquad s\in[t,T],
\end{equation}
we obtain
$X_s=x+\int_t^s\alpha_r\ dr+\int_t^s\sigma_r\ \d W_r,
\ t\le s\le T,\  \P$-a.s., where $W=(W_s)_{s\in[t,T]}$ is a $d$-dimensional $\P$-Brownian motion on $[t,T]$.

Fix $y\in\R^d$ and define the auxiliary process
$Y_s^y:=y+\int_t^s\alpha_r\ \d r+ W_s,\  s\in[t,T]$.
Then
\begin{equation}
\E^\P|Y_T^y|^2
\le
3|y|^2+3(T-t)\E^\P\int_t^T |\alpha_r|^2dr+3d(T-t)<\infty.
\end{equation}
Since $\hat\phi$ is finite and $\beta$--convex and $\hat\phi(0)=0$, we have for all $p\in\partial\hat\phi(0)$:
\begin{equation}
\hat\phi(z)\ge p\cdot z-\frac\beta2|z|^2
\ge -|p|\ |z|-\frac\beta2|z|^2,\  z\in\R^d.
\end{equation}
Since $\E^\P|Y_T^y|^2<\infty$, this implies
$\hat\phi(Y_T^y)^-\in\L_1(\P)$. Let $\Omega_{t,T}:=C([t,T];\R^d)$, let $R^{t,y}$ be Wiener measure on $\Omega_{t,T}$
started from $y$ at time $t$, and let
$\Q^y:=\Law_\P(Y^y)\in\Pc(\Omega_{t,T})$.
We claim that
\begin{equation}\label{entropylessthan}
H(\Q^y\ |\ R^{t,y})\le \frac12\E^\P\int_t^T |\alpha_r|^2 \d r.
\end{equation}
For this, let $n\ge1$ and define
$\tau_n:=\inf\Big\{s\ge t:\int_t^s |\alpha_r|^2 \d r\ge n\Big\}\wedge T,
\ \alpha_r^{(n)}:=\alpha_r\ 1_{\{r\le\tau_n\}}$,
and
$Y_s^{(n),y}:=y+\int_t^s\alpha_r^{(n)}\d r+W_s,\ s\in[t,T]$.
Set
$Z_n:=e^{
-\int_t^T \alpha_r^{(n)}\cdot \d W_r
-\frac12\int_t^T |\alpha_r^{(n)}|^2 \d r}$.
Since $\int_t^T |\alpha_r^{(n)}|^2 \d r\le n$, Novikov's criterion holds, so $Z_n$ is a true
martingale and
$\d\widetilde \P_n:=Z_n\ \d\P$
defines a probability measure. Under $\widetilde \P_n$, the process
$\widetilde W^{(n)}:=W+\int_t^. \alpha_r^{(n)} \d r$
is a Brownian motion on $[t,T]$.
Therefore, if $\Q_n:=\Law_\P(Y^{(n),y})$, then
$\Law_{\widetilde \P_n}(Y^{(n),y})=R^{t,y}$.
By contraction of relative entropy under the measurable map $\omega\mapsto Y^{(n),y}(\omega)$,
\begin{equation}
H(\Q_n\ |\ R^{t,y})
\le
H(\P\ |\ \widetilde \P_n)
=
\E^\P\!\left[\log\frac{d\P}{d\widetilde \P_n}\right]
=
\E^\P[-\log Z_n]
=
\frac12\E^\P\int_t^T |\alpha_r^{(n)}|^2 \d r,
\end{equation}
because $\E^\P[\int_t^T \alpha_r^{(n)}\cdot \d W_r]=0$. Hence
$H(\Q_n\ |\ R^{t,y})\le \frac12\E^\P\int_t^T |\alpha_r^{(n)}|^2 \d r$.

Also,
$\sup_{s\in[t,T]}|Y_s^{(n),y}-Y_s^y|
\le \int_{\tau_n}^T |\alpha_r| \d r$,
thus 
$\E^\P\Big[\sup_{s\in[t,T]}|Y_s^{(n),y}-Y_s^y|^2\Big]
\le
(T-t)\E^\P\int_{\tau_n}^T |\alpha_r|^2 \d r\longrightarrow 0$.
Thus $\Q_n\Rightarrow \Q^y$ weakly on $\Omega_{t,T}$. By lower semicontinuity of relative
entropy,
$H(\Q^y\ |\ R^{t,y})
\le \liminf_{n\to\infty} H(\Q_n\ |\ R^{t,y})
\le \frac12\E^\P\int_t^T |\alpha_r|^2 \d r$.

For $m\ge1$, define
$f_m(\omega):=\hat\phi(\omega_T)\wedge m,\  \omega\in\Omega_{t,T}$.
Then $f_m^+\le m$ and
$f_m^-= \hat\phi(\omega_T)^-$,
so $f_m\in \L^1(\Q^y)$. Moreover,
$e^{f_m(\omega)}\le e^{\hat\phi(\omega_T)},\  \omega\in\Omega_{t,T}$,
hence
\begin{equation}
\E^{R^{t,y}}[e^{f_m}]
\le \E^{R^{t,y}}[e^{\hat\phi(\omega_T)}]
=
(\Nc_{T-t}*e^{\hat\phi})(y)
=
e^{u_{T-t}^{\hat\phi}(y)}
<\infty.
\end{equation}
Applying the Donsker--Varadhan variational formula to $f_m$ under $R^{t,y}$, we see that 
$$
u_{T-t}^{\hat\phi}(y)
\ge
\E^{\Q^y}[f_m]-H(\Q^y| R^{t,y})
\uparrow
\E^{\Q^y}[\hat\phi(Y^y_T)]-H(\Q^y| R^{t,y})
\ge
\E^{\Q^y}[\hat\phi(Y^y_T)]
-\frac12\E^\P\!\!\int_t^T\!\! |\alpha_r|^2\d r,
$$
by monotone convergence and \eqref{entropylessthan}. Next, we have
$\hat\psi(X_T)\le \hat\phi(Y_T^y)+\frac\beta2|X_T-Y_T^y|^2,
\ \P$-a.s.
As $\E^\P|X_T-Y_T^y|^2
= |x-y|^2+\E^\P\int_t^T |\sigma_r-I_d|^2 \d r$,
we get
\begin{align}
\E^\P\Big[\hat\psi(X_T)-\!\!\int_t^T\!\!\!\!c(\gamma_r) \d r\Big]
&\le
\E^\P[\hat\phi(Y_T^y)]-\frac12\E^\P\!\!\int_t^T\!\!\!\! |\alpha_r|^2\d r
+\frac\beta2\Big(\E^\P|X_T\!-\!Y_T^y|^2\!-\!\E^\P\!\!\int_t^T \!\!\!\!|\sigma_r\!-\!I_d|^2\d r\Big) 
\\
&\le
u_{T-t}^{\hat\phi}(y)+\frac\beta2|x-y|^2.
\end{align}
By the arbitrariness of $y\in\R^d$ and $\P\in \Pc_t(\delta_x)$, this shows that
\begin{equation}\label{eq:V-upper-final-step4}
V_t^{\hat\psi}(x)\le \hat v_t(x),
\end{equation}
which proves the claim.

\smallskip
\noindent 4(b). Since $\hat\phi$ is finite and $\beta$--convex, the map
$\hat\psi=\Tc_\beta^+[\hat\phi]$ is finite and $\beta$--concave. Moreover,
$\mathbf T_\beta^-[\hat\psi]
= \mathbf T_\beta^-\mathbf T_\beta^+[\hat\phi]
= \hat\phi$,
and by \eqref{hatpsiintegrability},
$\mu_0\Big(\mathbf T_\beta^+
\big[u_T^{\mathbf T_\beta^-[\hat\psi]}\big]\Big)
= \mu_0\big(\mathbf T_\beta^+[u_T^{\hat\phi}]\big)<\infty$.
Thus $\hat\psi\in\bar\Cc^{\conc}\subset\Cc^{\conc}$. Moreover, \eqref{eq:V-upper-final-step4} implies
$V_0^{\hat\psi}\le \hat v_0$, and therefore
$\mu_0(V_0^{\hat\psi})\le \mu_0(\hat v_0)<\infty$.
We may apply Lemma~\ref{lem:reduce0} together with \eqref{eq:V-upper-final-step4} to obtain
\begin{equation}\label{eq:hatpsi-final-step4}
\Vc(\mu_0,\mu_T)\ge \mu_T(\hat\psi)-\mu_0(V_0^{\hat\psi})
\ge \mu_T(\hat\psi)-\mu_0(\hat v_0)
= \mathfrak J(\hat\phi)
=\Vc_{\rm red}(\mu_0,\mu_T),
\end{equation}
by Step 3.
As $\Vc_{\rm red}(\mu_0,\mu_T)=\Vc(\mu_0,\mu_T)$ by Theorem~\ref{thm:main} (i), we get
$\mu_T(\hat\psi)-\mu_0(\hat v_0)
= \mathfrak J(\hat\phi)
= \Vc_{\rm red}(\mu_0,\mu_T)
= \Vc(\mu_0,\mu_T)$.
Combining this with \eqref{eq:hatpsi-final-step4}, we obtain
\begin{equation}
\mu_T(\hat\psi)-\mu_0(V_0^{\hat\psi})=\Vc(\mu_0,\mu_T),
\qquad
\mu_0(V_0^{\hat\psi})=\mu_0(\hat v_0),
\end{equation}
and then $V_0^{\hat\psi}=\hat v_0$, $\mu_0$--a.s. by \eqref{eq:V-upper-final-step4}, proving that $\hat\psi$ attains the supremum in $\bar\Vc(\mu_0,\mu_T)$.

\medskip
\noindent {\bf 5.} In order to prove the claimed convergence of $\mathbf{T}^+_\beta[\phi_{n_k}]$ in \eqref{Tbeta+phinconv}, fix $R>0$ and $s\in(0,\eta T)$,
where $\eta\in(0,\delta)$ is the parameter fixed in Step {\bf 1}. Then
\begin{align}
\big\|\mathbf{T}^+_\beta[\phi_{n_k}]-\mathbf{T}^+_\beta[\hat\phi]\big\|_{\L^\infty(B_R)}
&\le
\big\|\mathbf{T}^+_\beta[\phi_{n_k}]-\mathbf{T}^+_\beta[u_s^{\phi_{n_k}}]\big\|_{\L^\infty(B_R)}
\\
&\quad
+\big\|\mathbf{T}^+_\beta[u_s^{\phi_{n_k}}]-\mathbf{T}^+_\beta[u_s^{\hat\phi}]\big\|_{\L^\infty(B_R)}
\\
&\quad
+\big\|\mathbf{T}^+_\beta[u_s^{\hat\phi}]-\mathbf{T}^+_\beta[\hat\phi]\big\|_{\L^\infty(B_R)}.
\end{align}
We shall prove in Step {\bf 5} below that, for every fixed $s\in(0,\eta T)$,
\begin{equation}\label{usphiconv}
u^{\phi_{n_k}}_s \longrightarrow u^{\hat\phi}_s
\quad\mbox{and}\quad
\mathbf{T}^+_\beta[u^{\phi_{n_k}}_s]\longrightarrow \mathbf{T}^+_\beta[u^{\hat\phi}_s],
~~\mbox{locally uniformly in }\R^d.
\end{equation}
Hence, for every fixed $R>0$ and $s\in(0,\eta T)$,
\begin{equation}\label{task-1}
\limsup_{k\to\infty}
\big\|\mathbf{T}^+_\beta[\phi_{n_k}]-\mathbf{T}^+_\beta[\hat\phi]\big\|_{\L^\infty(B_R)}
\le
\big\|\mathbf{T}^+_\beta[u_s^{\hat\phi}]-\mathbf{T}^+_\beta[\hat\phi]\big\|_{\L^\infty(B_R)}
+
\theta_s^R,
\end{equation}
where
\begin{equation}\label{theta}
\theta_s^R
:=
\sup_{n\ge 1}
\big\|\mathbf{T}^+_\beta[u_s^{\phi_n}]-\mathbf{T}^+_\beta[\phi_n]\big\|_{\L^\infty(B_R)}.
\end{equation}
In Step {\bf 6} we show that
\begin{equation}
\theta_s^R\longrightarrow 0
~~\mbox{as }s\searrow 0,
\end{equation}
and the same argument, applied to $\hat\phi$, yields
\begin{equation}
\big\|\mathbf{T}^+_\beta[u_s^{\hat\phi}]-\mathbf{T}^+_\beta[\hat\phi]\big\|_{\L^\infty(B_R)}
\longrightarrow 0
~~\mbox{as }s\searrow 0.
\end{equation}
Together with \eqref{task-1}, this proves \eqref{Tbeta+phinconv}.

\medskip
\noindent {\bf 6.} In this step we justify \eqref{usphiconv}. Fix $s\in(0,\eta T)$ with $\eta\in(0,\delta)$ as in Step {\bf 1}, and set
\begin{equation}
\Phi:=\{\hat\phi,\phi_n,\ n\ge1\}.
\end{equation}
By Lemma~\ref{lem:bounds}~(ii), \eqref{eq:u_up_bound}, at time $0$, together with the boundedness of
$\big(U_\eta^{\phi_n}\big)_n$ from Step {\bf 1}, there exists a constant $C_\eta<\infty$ such that
\begin{equation}\label{eq:common-upper-bound}
\phi_n\le C_\eta+\frac{|.|^2}{2\eta T},~n\ge1,
~\mbox{and therefore}~
\widehat\phi\le C_\eta+\frac{|.|^2}{2\eta T},
\end{equation}
by \eqref{phinconv}.

In particular, as $s<\eta T$, we have $u_s^\phi(x)=\log\big(\Nc_s*e^\phi(x)\big)\in\R$ for all $x\in\R^d,\ \phi\in\Phi$. 
We split the remaining arguments into three parts.
\begin{itemize}
\item[(5a)] For $\phi\in\Phi$ and $R\ge1$, define
\begin{equation}
\zeta^\phi_R(s,x)
:=
\frac{\Nc_s*(e^\phi\1_{B_R^c})(x)}
     {\Nc_s*(e^\phi\1_{B_R})(x)}.
\end{equation}
Then there exist constants $C_s,a_s>0$, depending only on $(\beta,T,d,\eta,s)$, such that
\begin{equation}\label{eq:tail-ratio}
\sup_{\phi\in\Phi}\zeta^\phi_R(s,x)
\le
C_s\exp\!\Big(
-\frac14\Big(\frac1s-\frac1{\eta T}\Big)R^2+a_s|x|^2
\Big),
\qquad x\in\R^d,\ R\ge1.
\end{equation}
Indeed, $
\phi(y)\le C_\eta+\frac{|y|^2}{2\eta T}$ for all $\phi\in\Phi$ by \eqref{eq:common-upper-bound}, and using $
\frac{|x-y|^2}{2s}
=
\frac{|y|^2}{2\eta T}
+\frac{|x-y|^2}{2s}
-\frac{|y|^2}{2\eta T}
\ge
\frac{|y|^2}{2\eta T}
+\frac14\Big(\frac1s-\frac1{\eta T}\Big)|y|^2-a_s|x|^2,
$
for a suitable constant $a_s>0$, we obtain
\begin{align}
\Nc_s\!*\!(e^\phi\1_{B_R^c})(x)
\le
e^{C_\eta}
\int_{B_R^c}\!\!\!\frac{e^{\frac{|y|^2}{2\eta T}-\frac{|x-y|^2}{2s}}}{(2\pi s)^{\frac d2}}\,\d y
\le
C_s
e^{
-\frac14\Big(\frac1s-\frac1{\eta T}\Big)R^2+a_s|x|^2}.
\label{ineqeta1}
\end{align}
On the other hand, since $\Phi$ is locally bounded below on $B_1$, there exists $m_1<\infty$ such that $
\phi(y)\ge -m_1$ for all $y\in B_1,\ \phi\in\Phi.
$ Hence, for $R\ge1$ and every $\phi\in\Phi$,
\begin{equation}
\Nc_s*(e^\phi\1_{B_R})(x)
\ge
\Nc_s*(e^\phi\1_{B_1})(x)
\ge
(2\pi s)^{-\frac d2}e^{-m_1}\int_{B_1}e^{-\frac{|x-y|^2}{2s}}\,\d y
\ge
C_s'e^{-a_s'|x|^2},
\end{equation}
for some constants $C_s'>0$ and $a_s'\ge0$, independent of $\phi$ and $R$. Combining this with \eqref{ineqeta1}
yields \eqref{eq:tail-ratio}.

\item[(5b)] Fix $R>0$, and set
\begin{equation}
\gamma_{k,R}:=\|\phi_{n_k}-\hat\phi\|_{\L^\infty(B_R)}\longrightarrow 0
\qquad\mbox{as }k\to\infty
\end{equation}
by \eqref{phinconv}. For every fixed $x\in\R^d$, the estimate in (5a) implies that
\begin{equation}
\sup_{\phi\in\Phi}\zeta^\phi_R(s,x)\longrightarrow 0
\qquad\mbox{as }R\to\infty.
\end{equation}
We now write
\begin{align}
\Nc_s*e^{\phi_{n_k}}(x)
&\ge
e^{-\gamma_{k,R}}\Nc_s*(\1_{B_R}e^{\hat\phi})(x)
=
\frac{e^{-\gamma_{k,R}}}{1+\zeta^{\hat\phi}_R(s,x)}
\,\Nc_s*e^{\hat\phi}(x),
\\
\Nc_s*e^{\phi_{n_k}}(x)
&\le
\big(1+\zeta^{\phi_{n_k}}_R(s,x)\big)\Nc_s*(\1_{B_R}e^{\phi_{n_k}})(x)
\le
\big(1+\zeta^{\phi_{n_k}}_R(s,x)\big)e^{\gamma_{k,R}}\Nc_s*e^{\hat\phi}(x).
\end{align}
Sending $k\to\infty$ and then $R\to\infty$, we see that $u^{\phi_{n_k}}_s(x)\longrightarrow u^{\hat\phi}_s(x)$, for all $x\in\R^d$. 
Moreover, the positive-time semiconvexity estimate from Section~5 implies that, for this fixed $s$, the maps $u_s^{\phi_{n_k}}\!+\!\frac{\kappa(s)}{2}|.|^2$ are convex on $\R^d$ for all $k$, with $\kappa(s):=\frac{\beta}{1+\beta s}$. Since the limit is finite everywhere, Rockafellar \cite{Rockafellar}, Theorem~10.8, p.~90,
implies
\begin{equation}
u^{\phi_{n_k}}_s\longrightarrow u^{\hat\phi}_s
\qquad\mbox{locally uniformly in }\R^d.
\end{equation}

\item[(5c)] We next prove the convergence of the Moreau envelopes. For every $\phi\in\Phi$, define
\begin{equation}
G^\phi_s(y):=u_s^\phi(y)+\frac{\beta}{2}|x-y|^2.
\end{equation}
As $\Phi$ is locally bounded from below on $B_1$, the same argument as in (5a) gives a quadratic lower bound $
u_s^\phi\ge -c_s-a_s|.|^2$, $\phi\in\Phi$,
for some constants $c_s,a_s>0$ depending only on $(\beta,T,d,\eta,s)$. Hence, after possibly enlarging constants,
\begin{equation}\label{eq:coercive-step5}
u_s^\phi(y)+\frac{\beta}{2}|x-y|^2
\ge
-C_{s,r}+c_{s,r}|y|^2,
\qquad x\in B_r,\ \phi\in\Phi,
\end{equation}
for some constants $C_{s,r}>0$ and $c_{s,r}>0$.
In particular, for fixed $(s,r)$, the right-hand side tends to $+\infty$ as $|y|\to\infty$, uniformly in $x\in B_r$ and $\phi\in\Phi$.

On the other hand, by \eqref{eq:common-upper-bound} and $s<\eta T$, there exists $C_{s,r}'<\infty$ such that $
u_s^\phi(0)\le C_{s,r}'$ for all $\phi\in\Phi$. Therefore,
\begin{equation}
\mathbf{T}^+_\beta[u_s^\phi](x)
\le
u_s^\phi(0)+\frac{\beta}{2}|x|^2
\le
C_{s,r}' + \frac{\beta}{2}r^2,
\qquad x\in B_r,\ \phi\in\Phi.
\end{equation}
Combining this with \eqref{eq:coercive-step5}, we deduce that every minimiser $y$ of $
u_s^\phi(\cdot)+\frac{\beta}{2}|x-\cdot|^2$ for $x\in B_r$ and $\phi\in\Phi$ satisfies $|y|\le R_{s,r}$ for some $R_{s,r}<\infty$ independent of $\phi$ and $x$.
Thus
\begin{equation}
\mathbf{T}^+_\beta[u_s^\phi](x)
=
\inf_{|y|\le R_{s,r}}
\Big\{
u_s^\phi(y)+\frac{\beta}{2}|x-y|^2
\Big\},
\qquad x\in B_r,\ \phi\in\Phi.
\end{equation}
The local uniform convergence of $u^{\phi_{n_k}}_s$ established in {\rm (5b)} then implies
\begin{equation}
\mathbf{T}^+_\beta[u^{\phi_{n_k}}_s]
\longrightarrow
\mathbf{T}^+_\beta[u^{\hat\phi}_s],
~~\mbox{locally uniformly in }\R^d,
\end{equation}
which proves \eqref{usphiconv}.
\end{itemize}

\medskip
\noindent {\bf 7.} It remains to prove that, for the family $\Phi:=\{\phi_n,\ n\ge 1\}$, we have
\begin{equation}\label{theta-seq}
\theta_s^R
:=
\sup_{n\ge1}
\big\|\mathbf{T}^+_\beta[u_s^{\phi_n}]-\mathbf{T}^+_\beta[\phi_n]\big\|_{\L^\infty(B_R)}
\longrightarrow 0
\qquad\mbox{as }s\searrow 0,
\quad\mbox{for all }R>0.
\end{equation}
Fix $R>0$ and $\rho\in(0,1)$. Since $\Phi$ is locally bounded and $\beta$--convex, it is locally equi-Lipschitz, so there exists
a common Lipschitz constant $L=L_{R,\rho}$ on $B_{R+\rho}$. Denoting $\delta_s(\rho):=\P(|W_s|\ge \rho)$, we see that
\begin{equation}\label{step4ineq1}
e^{u^\phi_s(x)}
=
\Nc_s*e^\phi(x)
\ge
\Nc_s*(e^\phi\1_{B_\rho(x)})(x)
\ge
e^{\phi(x)-L\rho}\big(1-\delta_s(\rho)\big),
~
\phi\in\Phi,\ x\in B_R.
\end{equation}
On the other hand,
\begin{equation}
e^{u^\phi_s(x)}
\le
\Nc_s*(e^\phi\1_{B_\rho(x)})(x)
+
\Nc_s*(e^\phi\1_{B_\rho(x)^c})(x).
\end{equation}
The first term is bounded above by $
\Nc_s*(e^\phi\1_{B_\rho(x)})(x)\le e^{\phi(x)+L\rho}.$ For the second term, by the same Gaussian-tail argument as in Step {\bf 5}(a), using the fixed parameter $\eta$ from Step {\bf 1},
there exists a deterministic function $g_{R,\rho}(s)$, defined for $0<s<\eta T$, such that
\begin{equation}
g_{R,\rho}(s)\longrightarrow 0
~\mbox{as }s\searrow 0,
~\mbox{and}~\Nc_s*(e^\phi\1_{B_\rho(x)^c})(x)\le g_{R,\rho}(s),
~\phi\in\Phi,\ x\in B_R.
\end{equation}
Hence
\begin{equation}
e^{u^\phi_s(x)}
\le
e^{\phi(x)+L\rho}+g_{R,\rho}(s),
\qquad
\phi\in\Phi,\ x\in B_R.
\end{equation}
Together with \eqref{step4ineq1}, this yields
\begin{equation}
-L\rho+\log\big(1-\delta_s(\rho)\big)
\le
u^\phi_s(x)-\phi(x)
\le
\log\big(e^{\phi(x)+L\rho}+g_{R,\rho}(s)\big)-\phi(x)
\le
L\rho + C_{R,\rho} g_{R,\rho}(s),
\end{equation}
for some constant $C_{R,\rho}>0$, by the Lipschitz property of $\log$ on a bounded interval away from the origin,
using the local uniform boundedness of $\Phi$ on $B_R$. Since both $\delta_s(\rho)\to 0$ and $g_{R,\rho}(s)\to 0$ as $s\searrow 0$,
we obtain
\begin{equation}\label{theta0}
\limsup_{s\searrow 0}
\sup_{\phi\in\Phi}
\|u^\phi_s-\phi\|_{\L^\infty(B_R)}
\le
L\rho
\underset{\rho\to 0}{\longrightarrow} 0
~\mbox{for all }R>0.
\end{equation}
We next localize the minimizers in the definition of $\mathbf{T}_\beta^+[u_s^\phi]$. By the same lower-bound argument as in Step {\bf 5}(c),
for every $R>0$ there exists $R'>0$ such that, for all sufficiently small $s\in(0,\eta T)$, all $\phi\in\Phi$, and all $x\in B_R$,
every minimizer in the definition of $\mathbf{T}_\beta^+[u_s^\phi](x)$ belongs to $B_{R'}$. Denoting by $y_s^\phi(x)\in B_{R'}$ such a minimizer, we have
\begin{align}
\mathbf{T}_\beta^+[u_s^\phi](x)
&=
u_s^\phi\big(y_s^\phi(x)\big)+\frac{\beta}{2}\big|x-y_s^\phi(x)\big|^2
\\
&\le
\phi\big(y_s^\phi(x)\big)+\frac{\beta}{2}\big|x-y_s^\phi(x)\big|^2
+\|u_s^\phi-\phi\|_{\L^\infty(B_{R'})}
\\
&\le
\mathbf{T}_\beta^+[\phi](x)+\|u_s^\phi-\phi\|_{\L^\infty(B_{R'})}.
\end{align}
On the other hand, by Jensen's inequality and the $\beta$--convexity of $\phi$, we have $
u_s^\phi\ge \phi-\frac{\beta d}{2}s$, and therefore $
\mathbf{T}_\beta^+[u_s^\phi](x)\ge \mathbf{T}_\beta^+[\phi](x)-\frac{\beta d}{2}s.$
Combining the last two estimates, we obtain
\begin{equation}
\big|\mathbf{T}_\beta^+[u_s^\phi](x)-\mathbf{T}_\beta^+[\phi](x)\big|
\le
\|u_s^\phi-\phi\|_{\L^\infty(B_{R'})}+\frac{\beta d}{2}s,
~~ \phi\in\Phi,\ x\in B_R,
\end{equation}
for all sufficiently small $s\in(0,\eta T)$. Taking the supremum over $\phi\in\Phi$ and $x\in B_R$, and using \eqref{theta0} on $B_{R'}$,
we obtain \eqref{theta-seq}.
\ep

\section{Primal Attainment} \label{sec:primal} 

To prove existence of a primal optimiser $\widehat\P$ for the problem SBB$(\mu_0,\mu_T)$ we need the following additional regularity properties of the maps $\hat u,\hat v$. and the corresponding $\msY$.

\begin{lemma} \label{lem:primalattainment}
{\rm (i)} For $\mu_0$-a.e. $x$, the minimum in $\mathbf{T}_\beta^+[\hat u_T](x)$ is attained by a unique minimizer $\msY_0(x)$. Moreover, $\msY_0$ is Lipschitz on this full $\mu_0$-measure set; we fix a Lipschitz Borel extension, still denoted $\msY_0$. 
\\
{\rm (ii)} The minimum value $\hat\psi=\mathbf{T}_\beta^+[\hat\phi]$ is attained by a non-empty compact set of minimizers $\msY_T$.
\\
{\rm (iii)} For all $t<T$, we have $|\frac{\hat v}{w}|_{\L^\infty[0,t]}+|\frac{\nabla \hat u}{\sqrt{w}}|_{\L^\infty[t,T]}<\infty$.
\end{lemma}

\proof (i) By Step~{\bf 3(c)} in the proof of Theorem~\ref{thm:main} {\rm (ii-a)} in Section~\ref{sec:dualattainment}, we have
$\hat v_0=\mathbf{T}_\beta^+[\hat u_T]\in \L^1(\mu_0)$ and in particular
$\hat v_0(x)<\infty$ for $\mu_0$-a.e.\ $x$. Denote, for such $x$,
$F_x(y):=\hat u_T(y)+\frac\beta2|x-y|^2$ .
Moreover
$\hat u_T(x)+\frac{|x|^2}{2T}
= -\frac d2\log(2\pi T)
+ \log\!\int_{\R^d} e^{\hat\phi(y)-\frac{|y|^2}{2T}+\frac{x\cdot y}{T}}\,\d y$
is convex as a log-Laplace transform. Then
\begin{equation}\label{eq:hess_uT_final}
D^2\hat u_T\succeq -\frac1T\ I_d
~\mbox{and therefore}~
D^2F_x(y)\succeq \Big(\beta-\frac1T\Big)I_d\succ0,
\end{equation}
as $\beta T>1$. Thus $F_x$ attains a unique minimizer, and since $(x,y)\mapsto F_x(y)$ is Borel, the measurable maximum theorem yields a Borel map $\msY_0$.

For $x_i\in\R^d$ and $y_i:=\msY_0(x_i)$, $i=1,2$, the first-order condition for the minimization writes
$x_i=y_i+\beta^{-1}\nabla \hat u_T(y_i)$. Denote $\delta x:=x_1-x_2$, $\delta y:=y_1-y_2$, then by subtracting and taking the scalar product with $\delta y$, we obtain $
\delta x\cdot\delta y
=\!
|\delta y|^2+\frac1\beta\big(\nabla\hat u_T(y_1)\!-\!\nabla\hat u_T(y_2)\big)\!\cdot\! \delta y
\ge
(1-\frac1{\beta T})|\delta y|^2,
$
by \eqref{eq:hess_uT_final}. By the Cauchy-Schwarz inequality, this implies that $|\delta y|\le (1-\frac1{\beta T})^{-1}|\delta x|$, which proves the required Lipschitz continuity of $\msY_0$. 

\vspace{2mm}
\noindent (ii)  As $\hat\phi$ is $\beta$-convex, the map $g_0:=\hat\phi+\frac\beta2|.|^2$ is proper, l.s.c. and convex.
We have 
$\hat\psi(x)=
\frac\beta2|x|^2-g_0^*(\beta x)$, 
By the first duality result, the maximizing sequence may be chosen so that
$\psi_{n_k}:=\Tc_\beta^+[\phi_{n_k}]\in C_w$, hence each $\psi_{n_k}$ is finite on $\R^d$.
Together with the local uniform convergence $\psi_{n_k}\to \hat\psi=\Tc_\beta^+[\hat\phi]$, this implies that $\hat\psi$ is finite on all of $\R^d$.
Hence $g_0^*(p)<\infty$ for all $p\in\R^d$, and therefore $g_0$ is coercive, and the map $y\mapsto g_0(y)-\beta x\!\cdot\!y$ is coercive and l.s.c. for each fixed $x$ and, as such, attains its minimum with compact argmin set $\msY_T(x)$.

\vspace{2mm}
\noindent (iii) Fix now $\tau<T$ and let us prove for some constant $C_\tau$ that
\begin{align}
|\nabla \hat u_s(x)|\le C_\tau(1+|x|),
\qquad s\in[T-\tau,T],\ x\in\R^d,
\label{eq:strip_grad_u_final_step6}
\\
|\hat v_t(x)|\le C_\tau(1+|x|^2),
\qquad 0\le t\le \tau,\ x\in\R^d,~\mbox{for all}~\tau<T.
\label{eq:v_quadratic_strip}
\end{align}
Set
$\kappa_\tau:=\frac{\beta}{1+\beta(T-\tau)}$
and define
$G_s(x):=\hat u_s(x)+\frac{\kappa_\tau}{2}|x|^2,
\ s\in[T-\tau,T]$.
By \eqref{eq:semiconvex-u}, each $G_s$ is convex. Moreover,
\eqref{eq:positive-time-finite} implies
$G_s(x)\le C_\tau(1+|x|^2),
\  s\in[T-\tau,T],\ x\in\R^d$.
Since $s\mapsto \hat u_s(0)$ is continuous on $[T-\tau,T]$, we also have
$|G_s(0)|\le C_\tau,\  s\in[T-\tau,T]$.
Choose any $p_s\in \partial G_s(0)$. For every unit vector $e$,
$p_s\cdot e\le G_s(e)-G_s(0),
\ -p_s\cdot e\le G_s(-e)-G_s(0)$.
Hence
$|p_s|\le C_\tau,\  s\in[T-\tau,T]$.
By convexity,
$G_s(x)\ge G_s(0)+p_s\cdot x\ge -C_\tau(1+|x|)$.
Now fix $x\in\R^d$, $s\in[T-\tau,T]$, and a unit vector $e$, and set
$r:=1+|x|$.
Since $G_s$ is convex and differentiable,
$\nabla G_s(x)\cdot e\le \frac{G_s(x+re)-G_s(x)}{r},
\ -\nabla G_s(x)\cdot e\le \frac{G_s(x-re)-G_s(x)}{r}$.
Using the quadratic upper bound at $x\pm re$ and the affine lower bound at $x$, we obtain
$|\nabla G_s(x)\cdot e|
\le \frac{C_\tau(1+|x\pm re|^2)+C_\tau(1+|x|)}{r}$.
Since $r=1+|x|$ and $|x\pm re|\le |x|+r\le 1+2|x|$, this yields
$|\nabla G_s(x)\cdot e|\le C_\tau(1+|x|)$.
Taking the supremum over $|e|=1$ yields
$|\nabla G_s(x)|\le C_\tau(1+|x|)$.
Since
$\nabla \hat u_s(x)=\nabla G_s(x)-\kappa_\tau x$,
we obtain \eqref{eq:strip_grad_u_final_step6}.

It remains to prove \eqref{eq:v_quadratic_strip}. For the upper bound, choose $y=0$ in the Moreau envelope defining
$\hat v_t=\mathbf T_\beta^+[\hat u_{T-t}]$. Since $T-t\in[T-\tau,T]$ for $0\le t\le\tau$, and since $s\mapsto \hat u_s(0)$ is bounded above on $[T-\tau,T]$, we get
$\hat v_t(x)
\le \hat u_{T-t}(0)+\frac{\beta}{2}|x|^2
\le C_\tau+\frac{\beta}{2}|x|^2,
\  0\le t\le \tau$.

For the lower bound, use the Hessian estimate from Lemma~\ref{lem:HC_Moreau}{\rm (3-a)}:
$D^2\hat v_t \succeq -\frac1{T-t}I_d \succeq -\frac1{T-\tau}I_d,
\  0\le t\le\tau$.
Set
$H_t(x):=\hat v_t(x)+\frac{|x|^2}{2(T-\tau)},\  0\le t\le\tau$.
Then $H_t$ is convex. Since $\hat v_t(0)=\mathbf T_\beta^+[\hat u_{T-t}](0)$ is bounded from below on $[0,\tau]$, and since the subgradients of $H_t$ at $0$ are bounded on $[0,\tau]$ by the same local convexity argument used above, there exists $C_\tau<\infty$ such that
$H_t(x)\ge -C_\tau(1+|x|),
\  0\le t\le\tau,\ x\in\R^d$.
Consequently,
$\hat v_t(x)\ge -C_\tau(1+|x|^2),
\ 0\le t\le\tau,\ x\in\R^d$.
Together with the upper bound, this proves \eqref{eq:v_quadratic_strip}.
\ep

\vspace{5mm}

\noindent {\bf Proof of Theorem \ref{thm:main} (ii-b-c-d)}
\smallskip \noindent \textit{{\bf 1.}}
Set
$m_0:={\msY_0}_\#\mu_0$ and $\nu_0(dy):=e^{-\hat u_T(y)}m_0(dy)$.
Since $\hat v_0=\mathbf T_\beta^+[\hat u_T]$ is finite $\mu_0$-a.e. and $\msY_0$ is the corresponding minimizer, the measure $\nu_0$ is well-defined, non-zero, and $\sigma$-finite. Define
\begin{equation}\label{eq:def_nu0_nuT_mT_final}
\nu_T:=\nu_0*\Nc_T,
\qquad m_T(dy):=e^{\hat\phi(y)}\nu_T(dy).
\end{equation}
Then
$m_0=e^{\hat u_T}\nu_0,
\ m_T=e^{\hat\phi}\nu_T$.
Since $\mu_0\in\Pc_2(\R^d)$ and $\msY_0$ is Lipschitz by Lemma \ref{lem:primalattainment} (i), it follows that
$m_0={\msY_0}_\#\mu_0\in\Pc_2(\R^d)$. Moreover, by Tonelli's theorem and the definition of $\hat u_T$,
\begin{equation}
m_T(\R^d)
= \nu_T(e^{\hat\phi})
= \nu_0(\Nc_T*e^{\hat\phi})
= m_0(e^{\hat u_T}e^{-\hat u_T})
= 1.
\end{equation}
Hence $m_T\in\Pc(\R^d)$. Since $\nu_T$ is absolutely continuous with respect to $\Leb$ by Gaussian convolution, possibly with an extended density, and since $m_T=e^{\hat\phi}\nu_T$ is a finite measure, we also have $m_T\ll \Leb$.

\smallskip \noindent \textit{{\bf 2.}}
Fix $g\in C_b(\R^d)$, let
$\hat\phi_\varepsilon:=\hat\phi+\varepsilon g,
\ \hat\psi_\varepsilon:=\Tc_\beta^+[\hat\phi_\varepsilon]$,
$\tilde\phi_\varepsilon:=\Tc_\beta^-\!\big(\Tc_\beta^+[\hat\phi_\varepsilon]\big),
\ \bar\phi_\varepsilon:=\tilde\phi_\varepsilon-\tilde\phi_\varepsilon(0)$.
By the Moreau envelope property, $\tilde\phi_\varepsilon$ is $\beta$-convex,
$\tilde\phi_\varepsilon\le \hat\phi_\varepsilon,
\ \Tc_\beta^+[\tilde\phi_\varepsilon]
= \Tc_\beta^+[\hat\phi_\varepsilon]$.
Since $g$ is bounded and $\hat\phi\in\Cc_w^{\rm conv}$, this gives
$\bar\phi_\varepsilon\in \Cc^{\rm conv}_{w}$ and $\bar\phi_\varepsilon(0)=0$.
Moreover $\mathfrak{J}$ is invariant under addition of constants, so
$\mathfrak{J}(\bar\phi_\varepsilon)=\mathfrak{J}(\tilde\phi_\varepsilon)$.
Finally, as $\tilde\phi_\varepsilon\le \hat\phi_\varepsilon$ and
$\Tc_\beta^+[\tilde\phi_\varepsilon]=\Tc_\beta^+[\hat\phi_\varepsilon]$, we have
$\mathfrak{J}(\tilde\phi_\varepsilon)
\ge \mathfrak{J}(\hat\phi_\varepsilon)$.
Since $\hat\phi$ maximizes $\mathfrak{J}$, we obtain, for all small enough $|\varepsilon|$,
\begin{equation}\label{eq:perturb_opt_final}
\mathfrak{J}(\hat\phi_\varepsilon)\le \mathfrak{J}(\bar\phi_\varepsilon)=\mathfrak{J}(\tilde\phi_\varepsilon)\le \mathfrak{J}(\hat\phi).
\end{equation}

\smallskip
\noindent\emph{Terminal derivatives.}
By Lemma \ref{lem:primalattainment} (ii), for each $x\in\R^d$ the set $\msY_T(x)$ is nonempty and compact. Hence, by
Danskin's theorem,
$\partial_+\hat\psi_\varepsilon\big|_{\varepsilon=0}(x)=\min_{y\in\msY_T(x)}g(y),
\ \partial_-\hat\psi_\varepsilon\big|_{\varepsilon=0}(x)=\max_{y\in\msY_T(x)}g(y)$.
Moreover, since $g$ is bounded,
$\hat\phi(y)\!-\!|\varepsilon| |g|_\infty
\le \hat\phi_\varepsilon(y)
\le \hat\phi(y)\!+\!|\varepsilon||g|_\infty,
\ y\in\R^d$,
and therefore, by monotonicity of $\Tc_\beta^+$,
$\hat\psi(x)-|\varepsilon||g|_\infty
\le \hat\psi_\varepsilon(x)
\le \hat\psi(x)+|\varepsilon||g|_\infty,
\  x\in\R^d$.
Thus
$\left|
\frac{\hat\psi_\varepsilon(x)-\hat\psi(x)}{\varepsilon}
\right| \le |g|_\infty,
\  \varepsilon\neq0$. Then dominated convergence yields
$$
\partial_\eps\mu_T(\hat\psi_\eps) \big|_{\varepsilon=0^+}
= \int_{\R^d}\min_{y\in\msY_T(x)}g(y)\,\mu_T(dx),
\qquad \partial_\eps\mu_T(\hat\psi_\eps) \big|_{\varepsilon=0^-}
= \int_{\R^d}\max_{y\in\msY_T(x)}g(y)\,\mu_T(dx).
$$

\smallskip
\noindent\emph{Initial derivative.}
Define
$u_{\varepsilon,T}:=\log(\Nc_T*e^{\hat\phi+\varepsilon g}),
\ v_{\varepsilon,0}:=\Tc_\beta^+[u_{\varepsilon,T}]$.
By Lemma \ref{lem:primalattainment} (i), the minimizer $\msY_0(x)$ is unique for $\mu_0$-a.e.\ $x$, hence Danskin's theorem gives
$\partial_\varepsilon v_{\varepsilon,0}(x)\big|_{\varepsilon=0}
= \partial_\varepsilon u_{\varepsilon,T}(\msY_0(x))\big|_{\varepsilon=0}
\ \mu_0\text{-a.e.\ }x$.
Since $g$ is bounded, differentiation under the Gaussian integral is justified and yields
$G(y):=\partial_\varepsilon u_{\varepsilon,T}(y)\big|_{\varepsilon=0}
= \frac{\Nc_T*(g e^{\hat\phi})(y)}{h_T(y)},
\  |G(y)|\le |g|_\infty$.
Also,
$e^{-|\varepsilon||g|_\infty}(\Nc_T*e^{\hat\phi})(y)
\le (\Nc_T*e^{\hat\phi+\varepsilon g})(y)
\le e^{|\varepsilon||g|_\infty}(\Nc_T*e^{\hat\phi})(y)$,
so
$|u_{\varepsilon,T}(y)-u_{0,T}(y)|\le |\varepsilon||g|_\infty,
\ y\in\R^d$.
By monotonicity of $\Tc_\beta^+$, this implies
$|v_{\varepsilon,0}(x)-v_{0,0}(x)|\le |\varepsilon||g|_\infty,
\ x\in\R^d$.
Hence dominated convergence applied to the difference quotients gives
\begin{equation}
\partial_\varepsilon\mu_0(v_{\varepsilon,0})\Big|_{\varepsilon=0}
= \mu_0\big(G(\msY_0)\big)
= m_0(G)
= \nu_0\big(\Nc_T*(g e^{\hat\phi})\big)
= \nu_T\big(ge^{\hat\phi}\big)
= m_T(g),
\end{equation}
by \eqref{eq:def_nu0_nuT_mT_final} and Tonelli's theorem.

\smallskip
\noindent\emph{Conclusion.}
Divide \eqref{eq:perturb_opt_final} by $\varepsilon$ and let $\varepsilon\to0^\pm$. Using the derivative formulas above, we obtain, for every $g\in C_b(\R^d)$,
\begin{equation}\label{eq:sandwich_final}
\mu_T\big(\min_{\msY_T}g\big)
\le m_T(g)
\le \mu_T\big(\max_{\msY_T}g\big).
\end{equation}

\smallskip \noindent \textit{{\bf 2.}} Let $\Gamma:=\{(x,y):y\in\msY_T(x)\}$ be the graph of $\msY$. Since $m_T\ll \Leb$ and $\hat\phi$ is $\beta$-convex, $\hat\phi$ is differentiable $m_T$-a.e. At every differentiability point $y$ such that $(x,y)\in\Gamma$, then
$y$ minimizes $z\mapsto \hat\phi(z)+\frac\beta2|x-z|^2$, so the first-order optimality condition yields
$\nabla\hat\phi(y)+\beta(y-x)=0$. Thus
$x=\msX_T(y):=y+\beta^{-1}\nabla\hat\phi(y)$.
Therefore, for $m_T$-a.e.\ $y$, the section
$\Gamma^y:=\{x\in\R^d:(x,y)\in\Gamma\}$
is contained in the singleton $\{\msX_T(y)\}$. Our objective in this step is to show that
\begin{equation}\label{mu_T=Xc_diezmT}
\mu_T={\msX_T}_\# m_T.
\end{equation}
Since $\hat\phi$ is l.s.c. the graph $\Gamma:=\{(x,y):y\in\msY_T(x)\}$ is closed and $\Gamma=\{c_\Gamma=0\}$, where
$c_\Gamma(x,y):=1\wedge \dist((x,y),\Gamma)$ is bounded, continuous, and nonnegative.
By Kantorovich duality, using the standard notation $f\oplus h=f(x)+h(y)$, we have
\begin{equation}
0\le
\inf_{\pi\in\Pi(\mu_T,m_T)}\pi(c_\Gamma)
=
\sup_{(f,h)\in\msD}\mu_T(f)+m_T(h),
~~\msD:=\{(f,h)\in C_b\times C_b:\ f\oplus h\le c_\Gamma\big\}.
\end{equation}
For $(f,h)\in\msD$, we have
$\max_{y\in\msY_T(x)} h(y)\le -f(x),\  x\in\R^d$.
Applying the right-hand side inequality in \eqref{eq:sandwich_final} with $g=h$, we get $
m_T(h)
\le
\int \max_{y\in\msY_T(x)}h(y)\ \mu_T(\d x)
\le
-\mu_T(f)$, hence $\le 0$. By arbitrariness of $(f,h)\in\msD$, we deduce that
$\inf_{\pi\in\Pi(\mu_T,m_T)}\int c_\Gamma\ d\pi=0$.
Since the cost is bounded and continuous, an optimal coupling $\pi^\star\in\Pi(\mu_T,m_T)$ exists and satisfies $\pi^\star(\Gamma)=1$, as the minimum value is $0$.

Disintegrate $\pi^\star$ with respect to its second marginal $m_T$:
$\pi^\star(\d x,\d y)=\pi^\star_y(\d x)\ m_T(\d y)$.
Since $\pi^\star(\Gamma)=1$, for $m_T$-a.e.\ $y$ the measure $\pi^\star_y$ is supported on
$\Gamma^y$, hence on $\{\msX_T(y)\}$. Thus
$\pi^\star_y=\delta_{\msX_T(y)}
\ \text{for }m_T\text{-a.e.\ }y$.
Taking first marginals, we obtain \eqref{mu_T=Xc_diezmT}.

\smallskip \noindent \textit{{\bf 3.}}
Let $\mathbf R$ be the law of a Brownian motion $Y$ on $[0,T]$ with initial law $m_0$.
Define
$Z_T:=e^{\hat\phi(Y_T)-\hat u_T(Y_0)}$.
Then
$\E^{\mathbf R}[Z_T]
= \E^{\mathbf R}\Big[ e^{-\hat u_T(Y_0)}
\E^{\mathbf R}\big[e^{\hat\phi(Y_T)}\mid Y_0\big]\Big] =1$.
We may therefore introduce the equivalent probability measure
$d\hat \Q:=Z_T\,d\mathbf R$. For any bounded Borel $f$, we have
$\E^{\hat \Q}[f(Y_T)]
=\E^{\mathbf R}\!\big[f(Y_T)e^{\hat\phi(Y_T)-\hat u_T(Y_0)}\big]
=m_T(f)$.
Thus
$Y_0\sim m_0
~\mbox{and}~Y_T\sim m_T~\mbox{under}~\hat\Q$.
The density process is
$Z_t:=\E^{\mathbf R}[Z_T|\mathcal F_t]
= e^{\hat u_{T-t}(Y_t)-\hat u_T(Y_0)},
\  0\le t<T$.
As $(t,y)\mapsto \hat u_{T-t}(y)$ solves the backward heat equation, It\^o's formula gives
$dZ_t=Z_t\nabla\hat u_{T-t}(Y_t)\cdot dW_t^{\mathbf R},
\  t<T$.
It follows from Girsanov's theorem that, under $\hat\Q$,
\begin{equation}\label{eq:Y_SDE_global}
Y_t
=Y_0+\int_0^t \nabla\hat u_{T-s}(Y_s)\,ds+W_t^{\hat\Q},
\qquad 0\le t<T.
\end{equation}
Define, for $0\le t<T$, $X_t:=Y_t+\frac1\beta\nabla\hat u_{T-t}(Y_t)$. Then $X_0=\msX_0(Y_0)$, so $X_0\sim\mu_0$ under $\hat \Q$, and $X$ has continuous paths on $[0,T)$ by the regularity of $\hat u$ on positive heat times. Moreover, for every $\tau<T$, since $T-s\in[T-\tau,T]$ for $0\le s\le\tau$ and
$\nabla\hat u$ has linear growth on this heat-time strip by Lemma~\ref{lem:primalattainment}{\rm (iii)}, the BDG inequality and Gronwall's lemma yield
\begin{equation}\label{eq:X_global_second_moment}
\E^{\hat \Q}\!\Big[\sup_{0\le s\le \tau}|Y_s|^2\Big]
+ \E^{\hat \Q}\!\Big[\sup_{0\le s\le \tau}|X_s|^2\Big]
<\infty, \qquad \tau<T.
\end{equation}

\smallskip \noindent \textit{{\bf 5.}}
Recall that
$\hat v_t=\mathbf T_\beta^+[\hat u_{T-t}],
\ \msX_t(y)=y+\frac1\beta\nabla\hat u_{T-t}(y),
\ X_t=\msX_t(Y_t), \quad t<T$.
Applying It\^o's formula to $X_t=\msX_t(Y_t)$ on every compact subinterval of $[0,T)$, and using the equation satisfied by $\hat u$, gives
$dX_t=a_t\,dt+\sigma_t\,dW_t,
\  0\le t<T$,
where
\begin{equation}\label{eq:feedback_coeffs_global}
\gamma_t=(a_t,\sigma_t),
\qquad
a_t=\nabla\hat v_t(X_t),
\qquad
\sigma_t=\big(I_d-\beta^{-1}D^2\hat v_t(X_t)\big)^{-1}.
\end{equation}
These are precisely the maximizers of the Hamiltonian $H(\nabla\hat v_t,D^2\hat v_t)$.
Fix $n\ge1$, $\tau<T$, and define
$\rho_n:=\tau\wedge\inf\{t\le \tau:\ |X_t|\ge n\}$.
Fix also $0<\varepsilon<\tau$. Since $\hat v\in C^{1,2}(Q_T)$, It\^o's formula on
$[\varepsilon\wedge\rho_n,\rho_n]$ gives
\begin{align}
\hspace{-5mm}
\E^{\hat\Q}\big[\hat v_{\rho_n}(X_{\rho_n})
                         -\hat v_{\varepsilon\wedge\rho_n}(X_{\varepsilon\wedge\rho_n})
                   \big]
&=
\E^{\hat\Q}\int_{\varepsilon\wedge\rho_n}^{\rho_n}\!\!\!
\Big(\partial_t\hat v+a_t\!\cdot\!\nabla\hat v+\tfrac12\sigma_t\sigma_t^\top\!:\!D^2\hat v\Big)(t,X_t)\d t
\\
&=\E^{\hat \Q}\!\int_{\varepsilon\wedge\rho_n}^{\rho_n} c(\gamma_t)\ \d t.
\label{eq:stopped_cost_identity_global}
\end{align}
Since $\beta T>1$, choose $\delta>0$ so small that
$\beta-(T-\delta)^{-1}>0$. Then for $t\in[0,\delta]$ and $x$ in a fixed
compact set $K\subset\R^d$, the function
$y\mapsto \hat u_{T-t}(y)+\frac{\beta}{2}|x-y|^2$
is uniformly strongly convex, hence has a unique minimizer. Moreover, a Taylor
expansion at $y=0$ shows that these minimizers remain in a compact ball $B_R$,
uniformly for $t\in[0,\delta]$ and $x\in K$. Since
$\hat u_{T-t}\to \hat u_T$ uniformly on $B_R$ as $t\downarrow0$, it follows that
\begin{equation}
\hat v_t(x)=\inf_y\Big(\hat u_{T-t}(y)+\frac{\beta}{2}|x-y|^2\Big)\to
\inf_y\Big(\hat u_T(y)+\frac{\beta}{2}|x-y|^2\Big)=\hat v_0(x),
~\mbox{uniformly on}~K.
\end{equation}

Letting $\varepsilon\downarrow0$, local uniform convergence of $v_t$ to $v_0$ near $t=0$
and continuity of $X$ give
$\hat v_{\varepsilon\wedge\rho_n}(X_{\varepsilon\wedge\rho_n})\to\hat v_0(X_0)
\ \hat \Q\text{-a.s.}$
Moreover the convergence is dominated by an integrable random variable, since on
$\{\rho_n>0\}$ the argument stays in the compact set $[0,T-\frac1n]\times \overline B_n$,
while on $\{\rho_n=0\}$ the term is exactly $\hat v_0(X_0)\in \L^1(\hat \Q)$.
We then deduce from \eqref{eq:stopped_cost_identity_global} that $
\E^{\hat \Q}\!\big[\hat v_{\rho_n}(X_{\rho_n})\big]
= \E^{\hat \Q}[\hat v_0(X_0)]
+ \E^{\hat \Q}\!\int_0^{\rho_n} c(\gamma_t)\ dt,
$ and by combining the quadratic growth property of $\hat v$ in Lemma \ref{lem:primalattainment} with \eqref{eq:X_global_second_moment}, we get by dominated and monotone convergence:
\begin{equation}\label{step5v0}
\E^{\hat \Q}[\hat v_\tau(X_\tau)]
= \E^{\hat \Q}[\hat v_0(X_0)]
+ \E^{\hat \Q}\!\Big[\int_0^\tau c(\gamma_t)\ \d t\Big]
=
\mu_0(v_0)+\E^{\hat \Q}\!\Big[\int_0^\tau c(\gamma_t)\ \d t\Big],
\  \tau<T.
\end{equation}
In particular,
$\E^{\hat \Q}\!\int_0^\tau c(\gamma_t)\ \d t<\infty$ for all $\tau<T$.

\smallskip \noindent \textit{{\bf 6.}}
Define
$F_s(y):=\hat u_s(y)+\frac\beta2|y|^2,\  s>0,
$ and $F(y):=\hat\phi(y)+\frac\beta2|y|^2$.
Recall that
$F_s\to F
 $ locally uniformly on $\R^d\text{ as }s\downarrow0$ and
\begin{equation}\label{eq:X_as_gradient_global}
\beta X_t=\nabla F_{T-t}(Y_t),\qquad t<T.
\end{equation}
Since $Y$ has continuous paths and $Y_T\sim m_T$, we have $Y_t\to Y_T$ a.s.\ as
$t\uparrow T$. Because $m_T\ll\Leb$ and $F$ is finite convex, $F$ is
differentiable at $Y_T$ for $\hat \Q$-a.e.\ sample point.

Fix an arbitrary sequence $t_n\uparrow T$. Set $f_n:=F_{T-t_n}$, $f:=F$, $x_n:=Y_{t_n}$, and $x:=Y_T$. Since
$f_n\to f$ locally uniformly, $x_n\to x$ $\hat\Q$-a.s., each $f_n$ is differentiable,
and $f$ is differentiable at $x$ for $\hat\Q$-a.e.\ sample point, the standard stability
result for gradients of convex functions yields
$\nabla F_{T-t_n}(Y_{t_n})\to \nabla F(Y_T)$, $\hat\Q$-a.s.
Hence, by \eqref{eq:X_as_gradient_global},
$X_{t_n}
= \beta^{-1}\nabla F_{T-t_n}(Y_{t_n})
\to \beta^{-1}\nabla F(Y_T)
\  \hat \Q$-a.s.
By Step~4, for $m_T$-a.e.\ $y$ we have
$\msX_T(y)=\beta^{-1}\nabla F(y)$.
Since $Y_T\sim m_T$, it follows that
$X_{t_n}\to \msX_T(Y_T)
\  \hat \Q$-a.s.
Since the sequence $t_n\uparrow T$ was arbitrary, we conclude that
$X_t\to X_T:=\msX_T(Y_T)
\  \text{as }t\uparrow T,\ \hat \Q$-a.s.
Thus $X$ admits a continuous extension to $[0,T]$. As $Y_T\sim m_T$ under $\hat \Q$ and $\mu_T={\msX_T}_\# m_T$ by Step 4, it follows that
$\Law_{\hat \Q}(X_T)=\mu_T$. Consequently
$$
\hat \P:=\Law_{\hat \Q}(X)
\in \Pc(\mu_0,\mu_T).
$$
\textit{{\bf 7.}}
Fix $\tau\in(0,T)$ and let $(\hat \P^\omega)_\omega$ be a regular conditional probability distribution of $\hat \P$ given $\mathcal F_\tau$. We first observe that for
$\hat \P$-a.e.\ $\omega$, the shifted canonical process on $[\tau,T]$ under $\hat \P^\omega$
starts from $X_\tau(\omega)$ and is a continuous semimartingale with absolutely
continuous characteristics $(a_r,\sigma_r)_{r\in[\tau,T)}$.
Therefore, for $\hat \P$-a.e.\ $\omega$, the Step~4 estimate on the shifted
interval $[\tau,T)$ with horizon $T-\tau$ and initial point $X_\tau(\omega)$ give
\begin{equation}\label{eq:conditional_shifted_moment_bound_final}
\E^{\hat \P^\omega}\!\Big[\sup_{\tau\le s<T}|X_s|^2\Big]<\infty.
\end{equation}
Fix now $t\in(\tau,T)$, $n\ge 1$, and define
$\rho_n^{\tau,t}:=t\wedge \inf\{s\in[\tau,t]: |X_s|\ge n\}$.
By the same It\^o calculation as in Step~5, we get by conditioning with respect to $\mathcal F_\tau$ that for $\hat \P$-a.e.\ $\omega$,
\begin{align}
\hat v_\tau(X_\tau(\omega))
&= 
\E^{\hat \P^\omega}\!\big[\hat v_{\rho_n^{\tau,t}}(X_{\rho_n^{\tau,t}})\big]
- \E^{\hat \P^\omega}\!\Big[\int_\tau^{\rho_n^{\tau,t}} c(\gamma_r)\ \d r\Big].
\\
&\longrightarrow
\E^{\hat \P^\omega}[\hat v_t(X_t)]
- \E^{\hat \P^\omega}\!\left[\int_\tau^t c(\gamma_r)\ \d r\right]~~
\mbox{as}~n\to\infty,
\label{eq:conditional_strip_identity_final}
\end{align}
by dominated convergence, due to the quadratic bound
$|\hat v_r(x)|\le C_{\tau,t}(1+|x|^2),
\  r\in[\tau,t],\ x\in\R^d$ in Step 5, together with
$\E^{\hat \P^\omega}\!\Big[\sup_{\tau\le s\le t}|X_s|^2\Big]<\infty
\ \text{for }\hat \P\text{-a.e. }\omega$. 

Next we let $t\uparrow T$. Since Step~6 gives $X_t\to X_T$ $\hat \P^\omega$-a.s., and
$\hat v_t\to\hat\psi$ locally uniformly as $t\uparrow T$, we have
$\hat v_t(X_t)\to \hat\psi(X_T)
\  \hat \P^\omega$-a.s.
Also, choosing $y=0$ in the Moreau envelope yields
$\hat v_t(x)\le \hat u_{T-t}(0)+\frac{\beta}{2}|x|^2$.
For fixed $\tau<T$, the function $t\mapsto \hat u_{T-t}(0)$ is bounded above on $t\in[\tau,T)$, and
\eqref{eq:conditional_shifted_moment_bound_final} implies that
$\sup_{t\in[\tau,T)} \hat v_t(X_t)^+
\le C_\tau+\frac{\beta}{2}\sup_{\tau\le s<T}|X_s|^2
\in \L^1(\hat \P^\omega)
\ \text{for }\hat \P\text{-a.e. }\omega$.
Hence dominated convergence gives
$\E^{\hat \P^\omega}[\hat v_t(X_t)^+]\to \E^{\hat \P^\omega}[\hat\psi(X_T)^+]$.
For the negative parts, Fatou's lemma gives
$\E^{\hat \P^\omega}[\hat\psi(X_T)^-]
\le
\liminf_{t\uparrow T} \E^{\hat \P^\omega}[\hat v_t(X_t)^-]$.
Combining with the monotone convergence on the nonnegative
cost term, we then obtain from \eqref{eq:conditional_strip_identity_final} that for $\hat \P$-a.e.\ $\omega$, $\hat v_\tau(X_\tau(\omega))
+
\E^{\hat \P^\omega}\!\Big[\int_\tau^T c(\gamma_r)\ \d r\Big]
\le
\E^{\hat \P^\omega}[\hat\psi(X_T)]$, and then
$$
\E^{\hat \P}[\hat v_\tau(X_\tau)]
+
\E^{\hat \P}\!\Big[\int_\tau^T c(\gamma_r)\ \d r\Big]
\le
\E^{\hat \P}[\hat\psi(X_T)]
=
\mu_T(\hat\psi),
$$
by the tower property. By \eqref{step5v0}, this implies
$
\E^{\hat \P}\!\left[\int_0^T c(\gamma_r)\ \d r\right]
\le \mu_T(\hat\psi)-\mu_0(\hat v_0)=\mathfrak{J}(\hat\phi)$.
Since $\hat \P\in\Pc(\mu_0,\mu_T)$, we have
$\E^{\hat \P}\!\left[\int_0^T c(\gamma_r)\,\d r\right] \ge {\rm SBB}(\mu_0,\mu_T)
= \mathfrak{J}(\hat\phi)$,
where the last equality follows from the strong duality already proved and the optimality of $\hat\phi$.
Hence
$\E^{\hat \P}\!\left[\int_0^T c(\gamma_r)\ \d r\right]=\mathfrak{J}(\hat\phi)=\mathrm{SBB}(\mu_0,\mu_T)$ by the strong duality of Theorem~\ref{thm:main} established in Section~4.
\ep

\bibliographystyle{plain}

\bibliography{biblioSBB}

\end{document}